\documentclass[10pt,reqno,oneside]{amsart}
\usepackage[body={145mm,215mm}]{geometry}
\usepackage{amsthm,amssymb,amsmath}
\usepackage[english]{babel}
\usepackage{tikz}
\usepackage{epstopdf}

\newtheorem{trm}{Theorem}[section]
\newtheorem{lem}[trm]{Lemma}
\newtheorem{prop}[trm]{Proposition}
\newtheorem{cor}[trm]{Corollary}
\theoremstyle{definition}
\newtheorem{de}[trm]{Definition}
\newtheorem{rem}[trm]{Remark}

\newcommand{\LL}{\mathcal A}
\newcommand{\A}{\mathcal{L}}
\newcommand{\ran}{\operatorname{ran}}
\newcommand{\la}{\langle}
\newcommand{\ra}{\rangle}

\renewcommand{\d}{\,\mathrm{d}}

\newcommand{\e}{\mathrm{e}}

\renewcommand{\theta}{\vartheta}

\newcommand{\R}{\mathbb{R}}

\newcommand{\aq}{\Leftrightarrow}

\newcommand{\N}{\mathbb{N}}

\newcommand{\HH}{\mathcal H}

\newcommand{\inj}{\hookrightarrow}

\newcommand{\beq}{\begin{equation}}
\newcommand{\eeq}{\end{equation}}
\newcommand{\rom}[1]{\textrm{\tiny{\uppercase\expandafter{\romannumeral #1}}}}
\newcommand{\hess}{\operatorname{Hess}}

\begin{document}

\author[M. Mileti\'c]{Maja Mileti\'c} \address{Institute for Analysis and
 Scientific Computing, Technical University Vienna, Wiedner
 Hauptstra\ss{}e 8, 1040 Vienna, Austria}

\author[D. St\"urzer]{Dominik St\"urzer} \address{Institute for Analysis and
 Scientific Computing, Technical University Vienna, Wiedner
 Hauptstra\ss{}e 8, 1040 Vienna, Austria}

\author[A. Arnold]{Anton Arnold} \address{Institute for Analysis and
 Scientific Computing, Technical University Vienna, Wiedner
 Hauptstra\ss{}e 8, 1040 Vienna, Austria}
 
 \author[A. Kugi]{Andreas Kugi} \address{Automation and Control Institute, Technical University Vienna , Gu\ss{}hausstra\ss{}e 27-29, 1040 Wien, Austria}

\title[Stability of an EBB with a nonlinear dynamic feedback system]{Stability of an Euler-Bernoulli beam with a nonlinear dynamic feedback system}

\begin{abstract}
This paper is concerned with the stability analysis of a lossless Euler-Bernoulli beam that carries a tip payload which is coupled to a nonlinear dynamic feedback system. This setup comprises nonlinear dynamic boundary controllers satisfying the nonlinear KYP lemma as well as the interaction with a nonlinear passive environment. Global-in-time wellposedness and asymptotic stability is rigorously proven for the resulting closed-loop PDE--ODE system. The analysis is based on semigroup theory for the corresponding first order evolution problem. For the large-time analysis, precompactness of the trajectories is shown by deriving uniform-in-time bounds on the solution and its time derivatives.
\end{abstract}

\maketitle

\section{Introduction}

Let us consider a linear homogeneous Euler-Bernoulli beam, clamped at one end and with tip mass at the other free end. The state of the beam at time $t$ is described by its transverse deflection $u(t,x)$ from the zero-state, where
$x\in\lbrack0,L]$ is the longitudinal coordinate of the beam, see Figure~\ref{beam}. The well known PDE for the motion of the beam reads as
\begin{equation} \label{beam_PDE}
\rho u_{tt}(t,x) + \Lambda u^{\rom{4}}(t,x)  = 0,
\end{equation}
with the mass per unit length $\rho$ and the flexural rigidity $\Lambda$. The boundary conditions for the clamped end at $x=0$ are given by
\begin{equation} \label{beam_BC1}
u(t,0)=u'(t,0)=0,
\end{equation}
and for the free end at $x=L$, we have
\begin{subequations} \label{beam_BC2}
\begin{align}
Ju_{tt}'(t,L)+\Lambda u^{\prime\prime}(t,L) & =-\tau_{e}  \label{beam_BC2_a}\\
Mu_{tt}(t,L)-\Lambda u^{\prime\prime\prime}(t,L) & =-f_{e},  \label{beam_BC2_b}
\end{align}
\end{subequations}
where $J$ and $M$ denote the mass moment of inertia and the mass of the tip mass, respectively, and $-\tau_{e}$ and $-f_{e}$ describe the external torque and force acting on the tip mass. Here and in the following, the notation $u_{t}$ is used for the derivative with respect to the time variable $t$, and $u'$ for the $x$-derivative.

In literature, there exists a number of contributions dealing with the design of boundary controllers to stabilize this type of system. To mention but a few, in \cite{littman} the asymptotic stability was shown using semigroup formulation and applying the La Salle Invariance principle. To obtain stronger, exponential stability, frequency domain criteria \cite{Morgul2001}, Riesz basis property \cite{guo2002riesz}, \cite{guo2006riesz} or energy multiplier methods \cite{rao1995uniform}, \cite{conrad1998stabilization} were employed. In contrast to these works, which are mainly based on linear static and dynamic boundary controllers, this paper is concerned with the interaction of the Euler-Bernoulli beam (\ref{beam_PDE}) - (\ref{beam_BC2})  with a finite-dimensional nonlinear dynamic system. In particular, it is assumed that this system generates a reaction torque $\tau_{e}=\tau_{e,1}+\tau_{e,2}$ and a reaction force $f_{e}=f_{e,1}+f_{e,2}$, respectively. The reaction torque and force is composed of the response of a nonlinear spring-damper system
\begin{subequations} \label{env_BC_SD}
\begin{align}
\tau_{e,1}  & =d_{1}(u_{t}'(t,L))+k_{1}(u'(t,L)) \label{env_BC_SD1}\\
f_{e,1}  & =d_{2}(u_{t}(t,L))+k_{2}(u(t,L))\label{env_BC_SD2}
\end{align}
\end{subequations}
and the response of a finite-dimensional nonlinear system with state $z_{j} \in\mathbb{R}^{n_{j}},$ $j=1,2$,
\begin{subequations} \label{env_BC1}
\begin{align}
(z_{1})_{t}  & =a_{1}(z_{1})+b_{1}(z_{1})u_{t}'(t,L) \label{EBB_4}\\
\tau_{e,2}  & =c_{1}(z_{1})\end{align}
\end{subequations}
and
\begin{subequations} \label{env_BC2}
\begin{align}
(z_{2})_{t}  & =a_{2}(z_{2})+b_{2}(z_{2})u_{t}(t,L) \label{EBB_5}\\
f_{e,2}  & =c_{2}(z_{2}),
\end{align}
\end{subequations}
which constitutes a strictly passive map from the time derivative of the tip angle $u_{t}'(t,L)$ to the reaction torque $\tau_{e,2}$ and from the velocity of the tip position $u_{t}(t,L)$ to the reaction force $f_{e,2}$,
respectively. The functions $a_{j}$, $b_{j}$, $c_{j}$, $d_{j}$, and $k_{j}$, $j=1,2$ as well as their mathematical properties will be specified in detail in the next section. 

The motivation for the setup (\ref{beam_PDE}) - (\ref{env_BC2}) is as follows: In literature, when designing a boundary controller for the system (\ref{beam_PDE}) - (\ref{beam_BC2}), it is usually assumed that the external torque $\tau_{e}$ and force $f_{e}$ directly serve as
control inputs. In this case, it is well known that the system (\ref{beam_PDE}) - (\ref{beam_BC2}) can be stabilized (even exponentially) by a simple (strictly) positive linear static feedback, see, e.g., \cite{lgm}, \cite{JakobZwart2012}. However, in real practical applications the external torque $\tau_{e}$ and force $f_{e}$ must be generated by some (electromagnetic, hydraulic or pneumatic) actuators whose dynamics cannot be neglected in general. In contrast to the usual approach in literature, it is therefore assumed in this work that these actuators are not ideally controlled, meaning that they are not serving as ideal torque and force sources, respectively, but that they are controlled in such a way that the subordinate closed-loop systems of the actuators comprising the actuator dynamics and a corresponding feedback controller constitute finite-dimensional passive dynamical systems according to (\ref{env_BC1}) and (\ref{env_BC2}). In summary, the system  (\ref{beam_PDE}) - (\ref{env_BC2}) may be interpreted as a feedback interconnected system with the lossless Euler Bernoulli beam (\ref{beam_PDE}) - (\ref{beam_BC2}) in the forward path and the passive spring-damper system (\ref{env_BC_SD}) as well as the strictly passive system (\ref{env_BC1}), (\ref{env_BC2}) in the feedback path, see Figure~\ref{fb_inter}. It is well known that the feedback interconnection of passive systems preserves the passivity, see, e.g., \cite{schaft2000}. This fact is often exploited in the controller design, see, e.g., \cite{ortega2002interconnection}, \cite{ott2008passivity}, for the finite-dimensional case.
However, in the infinite-dimensional case the analysis is typically confined to linear systems, see, e.g., \cite{lgm}, \cite{KT05}, \cite{Villegas2009}, or very recently \cite{Ramirez2014}. Thus, with this work we want to take a first step towards an extension of the state of the art to the nonlinear case by still considering a linear PDE but allowing for a nonlinear ODE at the boundary.
%
\begin{figure}[ht]
	\includegraphics[trim = 49mm 211mm 44mm 40mm, clip, scale=0.9]{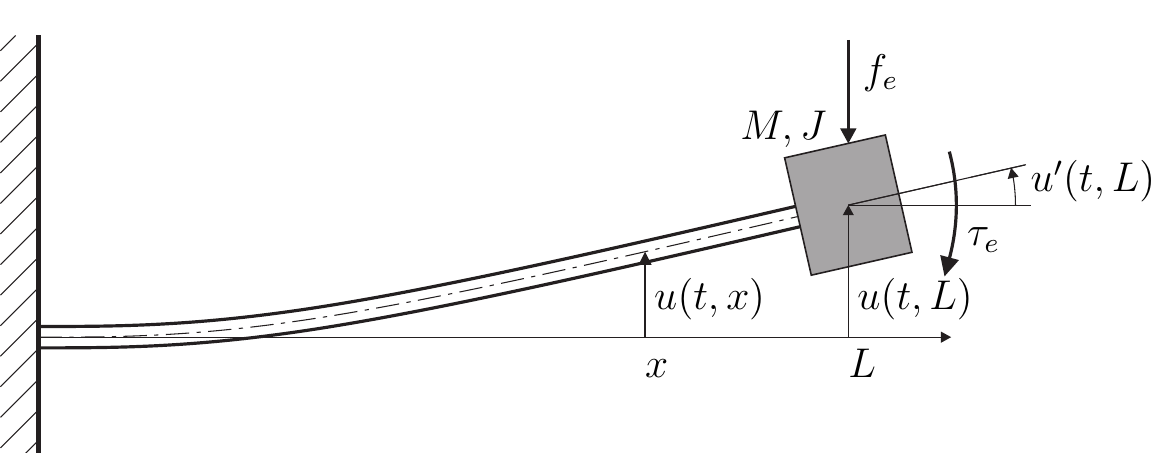}
		\caption{Euler-Bernoulli beam with tip mass.}
		\label{beam}
\end{figure}
%
%
%

\begin{figure}[ht]
	\includegraphics[trim = 20mm 130mm 30mm 78mm, clip, scale=0.9]{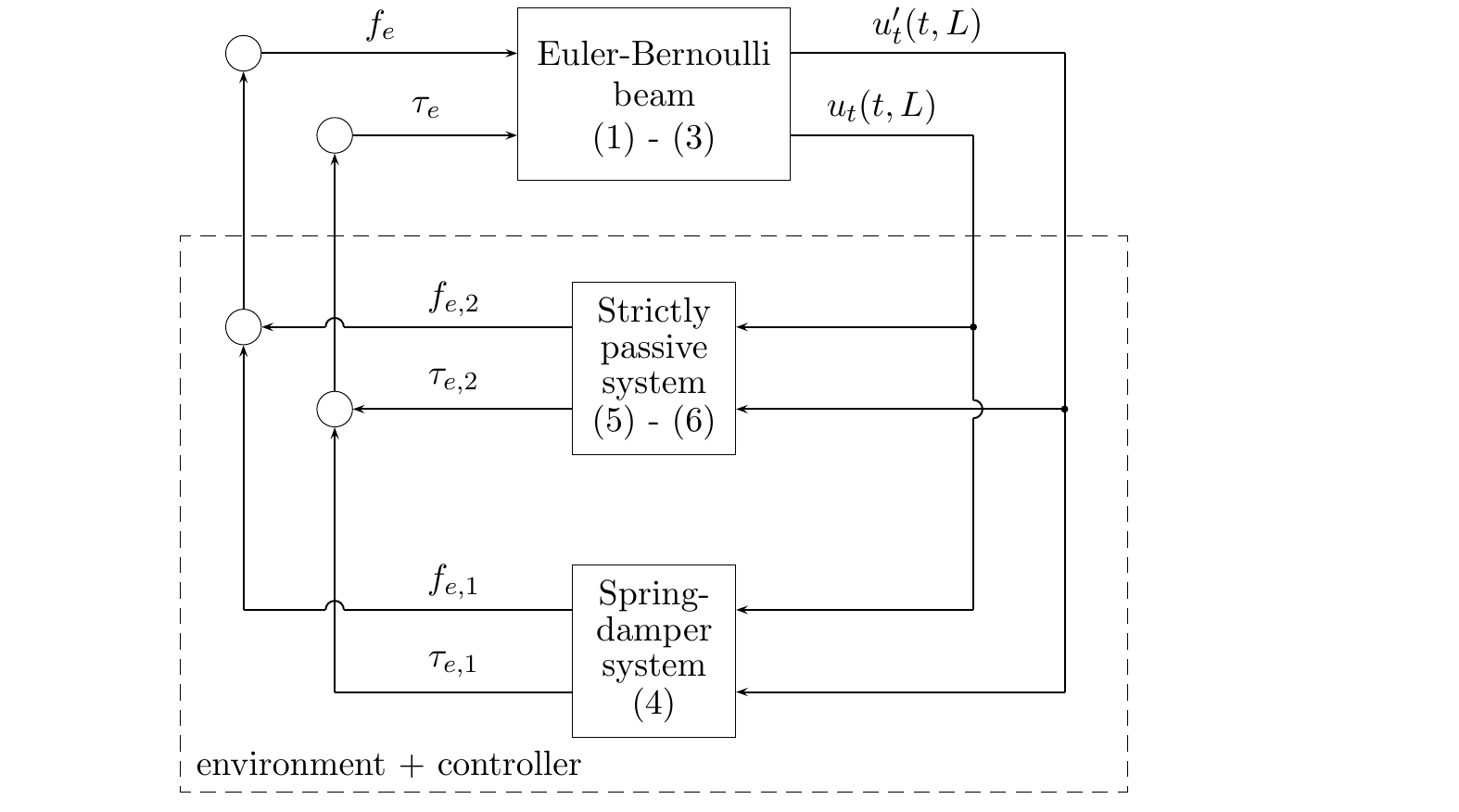}
		\caption{Interconnection of the Euler-Bernoulli beam system to a passive spring-damper system and a strictly passive system in the feedback path.}
		\label{fb_inter}
\end{figure}

The goal of this paper is to prove the global-in-time wellposedness and, most of all, the asymptotic stability of the feedback interconnected system (\ref{beam_PDE}) - (\ref{env_BC2}) according to Figure~\ref{fb_inter}. For both aspects, we have to deviate from the strategy employed in the analogous linear model (introduced and analyzed in \cite{KT05, MA}): In the linear case, the generator of the evolution semigroup is dissipative, which readily yields large-time solutions. The nonlinear semigroup for (\ref{beam_PDE}) - (\ref{env_BC2}) is \emph{not} dissipative (in the sense of \cite{cra_pa}). Hence, standard semigroup theory will first only yield local-in-time solutions, and the construction of an appropriate Lyapunov functional for (\ref{beam_PDE}) - (\ref{env_BC2}) then shows their global existence.

Asymptotic stability of the linear counterpart model is based on the precompactness of the trajectories, which can be obtained from the compactness of the resolvent for the generator. {For nonlinear evolution equations there exist different criteria for the precompactness of trajectories: They all split the generator of the nonlinear semigroup into the sum of a linear part $\mathcal L$ and a nonlinear part $\mathcal N$. In \cite{daf_sle}  $\mathcal A$ has to be maximal dissipative, and $\mathcal N$ has to be integrable along the solutions.  See \cite{cor_nov} for an application of this result, and also Section~\ref{S6} in this article below.  Another approach is due to \cite{Pazy3}, where only local integrability of $\mathcal N$ along trajectories is needed, however, the semigroup generated by $\mathcal L$ needs to be compact. Finally, in \cite{Webb} it is shown that the trajectories are precompact if the nonlinearity $\mathcal N$ is compact, and the semigroup generated by $\mathcal L$ is exponentially stable. Furthermore, we refer to \cite{zuyev} for further results regarding precompactness of trajectories. Unfortunately, none of the above results apply to the problem discussed in this paper, except for the special case $k_j=0$ discussed in Section~\ref{S6}. The reason for this is that the semigroup generated by the linear part is neither exponentially stable nor compact, and the nonlinearity $\mathcal N$ is generally not integrable along solutions. Hence,} for (\ref{beam_PDE}) - (\ref{env_BC2}), we shall follow a {different strategy, which was} devised for a simpler system in \cite{MSA} (it consists of a beam with a nonlinear spring and damper at the free end). For the precompactness of the trajectories of (\ref{beam_PDE}) - (\ref{env_BC2}), we shall here prove uniform $C^1$--bounds (w.r.t.\ time) on the solution, combined with compact Sobolev embeddings.

Note that the beam in (\ref{beam_PDE}) - (\ref{env_BC2}) (and in its linear counterpart) is undamped. Damping of the complete feedback system is only introduced via the damper of (\ref{env_BC_SD}) and the strictly passive systems (\ref{env_BC1}), (\ref{env_BC2}). This motivates that the linear model from \cite{MA} is asymptotically stable, but \emph{not} exponentially stable. Hence, exponential stability also cannot be expected for our nonlinear system (\ref{beam_PDE}) - (\ref{env_BC2}). Of course, exponential stability could be enforced by including damping terms into (\ref{beam_PDE}) (either a viscous damping of the form $+\alpha u_t$ or a Kelvin-Voigt damping of the form $+\alpha u_t^{\rom{4}}$). While viscous damping would lead to a simple extension of the subsequent analysis, the higher order derivatives in the Kelvin-Voigt damping would require a rather different mathematical setup. Hence, we shall not elaborate on such dampings here.

This paper is organized as follows: In Section \ref{secc5}, the technical assumptions made for the coefficients and functions of the system (\ref{beam_PDE}) - (\ref{env_BC2}) are specified, and in Section \ref{S3} the problem is formulated as a first order evolution equation. Using semigroup theory we prove in Section \ref{secc6} that it has a unique global-in-time solution. In Section \ref{S5}, the possible $\omega$--limit set of this evolution is derived and analyzed. For proving the asymptotic stability of (\ref{beam_PDE}) - (\ref{env_BC2}), we have to distinguish between two cases. For linear functions $k_j$, it is shown in Section \ref{S6} that asymptotic stability can be achieved for all mild solutions. For nonlinear $k_j$, it is much more involved to prove precompactness of the trajectories. In this case, asymptotic stability of classical solutions is shown in Section \ref{S7}.

\section{Preliminaries}\label{secc5}

{In the following sections, we will give a rigorous mathematical analysis of the feedback interconnected system (\ref{beam_PDE}) - (\ref{env_BC2}) according to Figure~\ref{fb_inter}. For this, the assumptions on the parameters and functions appearing in (\ref{beam_PDE}) - (\ref{env_BC2}) have to be specified. First of all, let us assume that the mass per unit length $\rho$, the flexural rigidity $\Lambda$, the mass moment of inertia $J$, and the mass $M$ of the tip mass are constant and positive.
For the spring-damper system (\ref{env_BC_SD}) we make the following assumptions for $j=1,2$:
\begin{description}
\item[\textbf{(A1)}] There holds $d_j\in C^2(\R;\R)$ (i.e. the space of two times continuously differentiable functions, see \cite{adams}), and\footnote{Here, $d_j'$ and $k_j'$ denote the derivative with respect to the variable $s$.}
\begin{subequations}\label{d:c}
\begin{align}
d_j(0)&=0,\\
d_j'(s)s &\ge 0,\quad\forall s\in\R,\label{d:c_1}\\
d_j'(0) &>0.
\end{align}
\end{subequations}
Note that this implies $d_j(s)s>0$ for all $s\neq 0$.
\item[\textbf{(A2)}] We require $k_j\in C^2(\R;\R)$, with $k_j'(0)>0$ and 
\begin{align}
V_{k_j}(s):=\int_0^s k_j(\sigma)\d \sigma& > 0,\quad\forall s\in\R\setminus\{0\}.\label{k:c_2}
\end{align}
\end{description}
Based on the assumptions (A1) - (A2), it can be easily shown that the spring-damper system (\ref{env_BC_SD}) is strictly passive from the inputs $u_{t}'(t,L)$ and $u_{t}(t,L)$ to the outputs $\tau_{e,1}$ and $f_{e,1}$, respectively, with the positive definite storage functions $V_{k_j},\ j=1,2$, according to (\ref{k:c_2}).}

{As a further consequence, we find uniquely determined constants $D_j,K_j>0$ and functions $\delta_j,\kappa_j\in  C^2(\R;\R)$ with $\delta_j(s)=\mathcal O (s^2)$ and $\kappa_j(s)=\mathcal O (s^2)$ for $s\to 0$ such that
\begin{align}
d_j(s)&=D_js+\delta_j(s),\quad\forall s\in\R,\label{d:c_2}\\
      k_j(s)&=K_js+\kappa_j(s),\quad\forall s\in\R.\label{k:c_1}
\end{align}
Hence, $D_js$ is the linearization of $d_j(s)$, and $K_js$ is the linearization of $k_j(s)$ around $s=0$.}

{
\begin{description}
\item[\textbf{(A3)}] Furthermore, we assume that there there exist (storage) functions $V_j\in C^2(\R^{n_j};\R)$, for $j=1,2$, such that for all $z_j\in\R^{n_j}$:
\begin{align}
V_j(0)=0,\quad V_j(z_j)&> 0,\quad (z_j\neq 0),\label{v:c_1}\\
\nabla V_j(z_j)\cdot a_j(z_j)&<0,\quad (z_j\neq 0),\label{a:c_3}\\
\nabla V_j(z_j)\cdot b_j(z_j) &= c_j(z_j)\label{c:c_3}.
\end{align}
\end{description}
According to the Kalman-Yakubovich-Popov (KYP) lemma for nonlinear systems with affine input, see {Lemma~4.4} in \cite{lozano2000}, this implies the strict passivity of the systems (\ref{env_BC1}) and (\ref{env_BC2}).}

{
For the mathematical analysis we furthermore require for ${j=1,2}$:
\begin{description}
\item[\textbf{(A4)}] Assume there holds \footnote{Note that condition \eqref{v:c_2} can be relaxed, see Remark~\ref{mark}.}
\begin{align}
\lim_{|z_j|\to\infty}V_j(z_j)&=\infty,\label{v:c_2}\\
P_j:=\hess(V_j)(0)&> 0.\label{v:c_3}
\end{align}
\item[\textbf{(A5)}]${a_j, \,b_j\in C^2(\R^{n_j}; \R^{n_j})}$ and
\begin{subequations}\label{14_10_1}
\begin{align}
a_j(0) &= 0,\label{14_10_1a}\\
\det A_j &\neq 0,\label{14_10_1b}
\end{align}
\end{subequations}
where $A_j= J_{a_j}(0)$ is the Jacobian of $a_j$ at $z_j=0$.
\item[\textbf{(A6)}] ${c_j\in C^2(\R^{n_j}; \R)}$, and
\begin{align*}
c_j(0) &=0.
\end{align*}
\end{description}
The assumptions (A3) - (A6) have the following implications, for $j=1,2$:
\begin{itemize}
 \item There exists a unique {\em regular} matrix $A_j\in\R^{n_j\times n_j}$ and a function $\alpha_j\in C^2(\R^{n_j}; \R^{n_j})$ such that for all $z_j\in\R^{n_j}$
\begin{subequations}\label{a:c}
 \begin{align}
      a_j(z_j)&=A_jz_j+\alpha_j(z_j),\label{a:c_1}\\
      |\alpha_j(z_j)| &= \mathcal O(|z_j|^2)\quad\text{ as }z_j\to 0,\label{a:c_2}
\end{align}
\end{subequations}
hence $A_j z_j$ is the linearization of $a_j(z_j)$ around the origin. By using the first order Taylor expansion of $\nabla V_j$ around the origin we conclude from \eqref{a:c_3} and \eqref{v:c_1} that
\begin{align}
z_j^\top(P_jA_j)z_j&\le 0,\quad \forall z_j\in\R^{n_j},\label{a:c_5}\\
\intertext{and from \eqref{v:c_3} and \eqref{14_10_1} we find}
|\nabla V_j(z_j)\cdot a_j(z_j)|&\ge C |z_j|^2\quad\text{ as }z_j\to 0,\label{a:c_4}
\end{align}
for some positive constant $C$.
\item There exists a unique vector $B_j\in\R^{n_j}$ and a function $\beta_j\in C^2(\R^{n_j}; \R^{n_j})$ such that for all $z_j\in\R^{n_j}$
\begin{subequations}\label{b:c}
 \begin{align}
      b_j(z_j)&=B_j+\beta_j(z_j),\label{b:c_1}\\
       |\beta_j(z_j) |&= \mathcal O(|z_j|)\quad\text{as }z_j\to 0.\label{b:c_2}
\end{align}
\end{subequations}
\item There exists a unique vector $C_j\in\R^{n_j}$ and a function $\gamma_j\in C^2(\R^{n_j}; \R)$ such that for all $z_j\in\R^{n_j}$
\begin{subequations}\label{c:c}
 \begin{align}
      c_j(z_j)&=C_j \cdot z_j+\gamma_j(z_j),\label{c:c_1}\\
      |\gamma_j(z_j)| &= \mathcal O(|z_j|^2)\quad\text{ as }z_j\to 0.\label{c:c_2}
      \intertext{Note that \eqref{c:c_3} implies}
       P_j B_j&=C_j.\label{c:c_4}
\end{align}
\end{subequations}
\end{itemize}
To illustrate that the above assumptions may be satisfied we just mention the linear model from \cite{KT05,MA} with $\alpha_j=\beta_j=\gamma_j=\delta_j=\kappa_j=0$. There (and in many nonlinear perturbations of it) assumptions (A1) - (A6) hold. }

{
\begin{rem}
For this paper it would even be possible to only make the weaker assumptions $a_j, b_j\in W_{\mathrm{loc}}^{2,\infty}(\R^{n_j};\R^{n_j})$, $c_j\in W_{\mathrm{loc}}^{2,\infty}(\R^{n_j};\R)$ and $d_j,k_j\in W_{\mathrm{loc}}^{2,\infty}(\R;\R)$, for $j=1,2$. Here, $W_{\mathrm{loc}}^{2,\infty}$ is the space of all $C^1$-functions whose first order derivatives are locally Lipschitz continuous functions (cf.~\cite{adams}). In particular, the local behavior of the functions $\alpha_j$, $\beta_j$, $\gamma_j$, $\delta_j$ and $\kappa_j$ around the origin also stays the same, which can be seen by using the integral form of the remainder in the respective Taylor expansions.
\end{rem}
For the rest of this paper the conditions (A1) - (A6) on the system \eqref{beam_PDE} - \eqref{env_BC2} are tacitly always assumed to hold.}

\section{Formulation as an Evolution Equation}\label{S3}

System (\ref{beam_PDE}) - (\ref{env_BC2}) is reformulated as an evolution equation in the Hilbert space
\begin{align*}
\HH=\{y=[u,v,z_1,z_2,\xi,\psi]^\top&: u\in\tilde H_{0}^2(0,L), v\in L^2(0,L),z_j\in\R^{n_j}, \xi,\psi\in\R\},
\end{align*}
where for $n\ge 2$ \[\tilde H_{0}^n(0,L):=\{f \in H^n(0,L) : f(0)=f'(0)=0\}.\]
{Note that we impose the function value and its first derivative only at the left boundary, i.e.~$x=0$. Hence, $\tilde H_0^n$ differs from the standard Sobolev spaces $H^n$ and $H_0^n$. We refer to \cite{adams} for the Lebesgue space $L^2(I)$ and the Sobolev spaces $H^n(I)$ on some interval $I$.} The inner product is defined by
\begin{align}
\begin{split}\label{ip_II}
 \la y,\tilde y\ra_\HH =\,& \Lambda \int_0^L u^{\prime\prime}\tilde u^{\prime\prime} \d x +\rho\int_0^L v\tilde v \d x+ \frac 1J \xi\tilde\xi+\frac 1M\psi\tilde\psi \\
      &+ K_1 u'(L)\tilde u'(L)+K_2 u(L)\tilde u(L)
      +  z_1^\top P_1\tilde z_1+z_2^\top P_2\tilde z_2 ,
\end{split}\end{align}
where the positive definite matrices $P_j$, for $j=1,2$, are due to \eqref{v:c_3}. For the following, the operator
\[\LL:\begin{bmatrix}
              u \\ v\\ z_1 \\ z_2 \\ \xi \\ \psi
             \end{bmatrix}\mapsto
\begin{bmatrix}
 v \\
-\frac{\Lambda}\rho u^\rom{4} \\
a_1(z_1)+\frac 1J b_1(z_1)\xi\\
a_2(z_2)+\frac 1M b_2(z_2)\psi\\
-\Lambda u^{\prime\prime}(L)-[c_1(z_1)+ d_1(\frac{\xi}{J}) +k_1 (u'(L))]\\
\Lambda u^{\prime\prime\prime}(L)-[c_2(z_2)+ d_2(\frac{\psi}{M}) +k_2 (u(L))]
\end{bmatrix}
\]
is introduced on the domain
\begin{align}
 D(\LL)=\{y\in\HH&:u\in \tilde H_{0}^4(0,L), v\in\tilde H_{0}^2(0,L),\label{da3}\\
 &\quad \xi=Jv'(L),\psi=Mv(L)\}.\nonumber
\end{align}
Based on the formulation of the coefficient functions, the operator $\LL$ is decomposed into a linear and a nonlinear part:

{\bf Linear part:} The linear part is denoted by ${\A}$, which is the linearization of $\LL$ around the origin:
      \[{\A}:\begin{bmatrix}
              u \\ v\\ z_1 \\ z_2 \\ \xi \\ \psi
             \end{bmatrix}\mapsto
\begin{bmatrix}
 v \\
-\frac{\Lambda}\rho u^\rom{4} \\
A_1z_1+\frac 1J B_1\xi\\
A_2z_2+\frac 1M B_2\psi\\
-\Lambda u^{\prime\prime}(L)-[C_1z_1+\frac 1J D_1\xi +K_1 u'(L)]\\
\Lambda u^{\prime\prime\prime}(L)-[C_2z_2+\frac 1M D_2\psi +K_2 u(L)]
\end{bmatrix},\]
and the domain is $D({\A})=D(\LL)$.

{\bf Nonlinear part:} {The nonlinear part $\mathcal N$ is defined as the following continuous operator on all of $\HH$:}
      \[\mathcal N:\begin{bmatrix}
              u \\ v\\ z_1 \\ z_2 \\ \xi \\ \psi
             \end{bmatrix}\mapsto
\begin{bmatrix}
 0 \\
0 \\
\alpha_1(z_1)+\frac 1J \beta_1(z_1)\xi\\
\alpha_2(z_2)+\frac 1M \beta_2(z_2)\psi\\
-\gamma_1(z_1)- \delta_1(\frac{\xi}{J})-\kappa_1(u'(L))\\
-\gamma_2(z_2)- \delta_2(\frac{\psi}{M})-\kappa_2(u(L))
\end{bmatrix}.\]
{On $D(\LL)$ there holds $\LL = {\A}+\mathcal N$.}
\bigskip

\begin{trm}\label{teo73}
The linear operator ${\A}$ with domain $D({\A})$ generates a $C_0$-semigroup of contractions in $\HH$, denoted by $(\e^{t{\A}})_{t\ge 0}$.
\end{trm}
\begin{proof}
This result has been shown for {the same operator} in Section 4.2 of \cite{KT05}. For convenience of the reader we briefly sketch the main steps of the proof. A brief calculation yields for $y\in D({\A})$, using  \eqref{a:c_5}:
\begin{align*}
\la {\A}y,y\ra_\HH = &z_1^\top(P_1A_1)z_1+z_2^\top(P_2A_2)z_2\\
&\quad -D_1|v'(L)|^2-D_2|v(L)|^2\le 0.
\end{align*}
Hence the operator ${\A}$ is dissipative in $\HH$ with respect to the inner product \eqref{ip_II}. Furthermore, the inverse ${\A}^{-1}$ exists and is bounded (even compact). Now the statement immediately follows from the Lumer-Phillips theorem.
\end{proof}

\begin{rem}\label{hypoer}
Since ${\A}$ is the infinitesimal generator of a $C_0$-semigroup of contractions, ${\A}$ is dissipative and {$\ran (\lambda-{\A})=\HH$ for all $\lambda>0$. In particular $\ran (I-{\A})=\HH$. So ${\A}$ is {\em hyper-dissipative} according to Definition~2.1 in \cite{cra_pa}. And Theorem~2.2 in \cite{cra_pa} shows that ${\A}$ is maximal dissipative, i.e.~${\A}$ is not contained in a strictly larger dissipative operator (in the sense of graphs). This property is needed for the proof of Theorem~\ref{fin_trm}.}
\end{rem}


\section{Existence of Solutions}\label{secc6}

We are interested in solutions of the following initial value problem in $\HH$:
\begin{subequations}\label{ivp}
 \begin{align}
   y_t(t)&=\LL y(t)={\A}y(t)+\mathcal Ny(t),\label{iv_1}\\
  y(0)&=y_0\in \HH\label{ic_1}.
 \end{align}
\end{subequations}
Any (mild) solution $y(t)\in C([0,T];\HH)$, for $T>0$, is known to satisfy Duhamel's formula:
      \begin{equation}\label{milde}
       y(t)=\e^{t {\A}}y_0+\int_0^t \e^{(t-s){\A}}\mathcal N y(s)\d s,\quad 0\le t< T.
      \end{equation}
%

\begin{prop}\label{pro_61}
For every $y_0\in \HH$, there exists some maximal $0<T_{\mathrm{max}}(y_0)\le\infty$ such that \eqref{ivp} has a unique mild solution $y(t)$ on $[0,T_{\mathrm{max}}(y_0))$. If $y_0\in D(\LL)$, the corresponding mild solution $y(t)$ is a classical solution. If $T_{\mathrm{max}}(y_0)<\infty$, then $\lim_{t\nearrow T_{\mathrm{max}}(y_0)} \|y(t)\|_\HH=\infty$.
\end{prop}

\begin{proof}
By assumption, the functions $\alpha_j,\beta_j,\gamma_j,\delta_j$ and $\kappa_j$ are continuously differentiable and {locally lipschitz continuous, so the nonlinear map $\mathcal N:\HH\to\HH$ has the same properties}. Furthermore, ${\A}$ is the generator of a $C_0$-semigroup, see Theorem~\ref{teo73}. Now we may apply Theorem~6.1.4 in \cite{pazy} to the autonomous problem \eqref{ivp}, which yields the existence of a unique mild solution on the {maximal} time interval $[0,T_{\mathrm{max}}(y_0))$. {If  $T_{\mathrm{max}}(y_0)<\infty $, then a blowup of $y(t)$ occurs.} Moreover, Theorem~6.1.5 in \cite{pazy} implies that for $y_0\in D(\LL)$, any mild solution is a classical solution.
\end{proof}

Next we introduce the functional $H:\HH\to\R$, given by
\begin{align*}
 H(y)&:=\frac 12\int_0^L \big(\Lambda |u^{\prime\prime}|^2+\rho |v|^2\big)\d x+\frac {|\xi|^2}{2J}+\frac{|\psi|^2} {2M}\\
 &\phantom{:=}+\underset{V_{k_{1}}\left(u^{\prime}(L)\right)}{\underbrace{\int_0^{u^{\prime}(L)}\!\!\!\!\!\!\!\!\!k_1(s)\d s}}+
\underset{V_{k_{2}}\left(u(L)\right)}{\underbrace{\int_0^{u(L)}\!\!\!\!\!\!\!\!\!k_2(s)\d s}}+V_1(z_1)+V_2(z_2).
\end{align*}
Note that the first integral term in $H(y)$ corresponds to the strain energy and kinetic energy of the Euler-Bernoulli beam, the next two summands are the translational and rotational part of the kinetic energy of the tip mass, $V_{k_{j}}$, for $j=1,2$, is the potential energy stored in the nonlinear spring elements, see (\ref{env_BC_SD}) and (\ref{k:c_2}), and $V_{j}$, for $j=1,2$, are the non-negative storage functions of the strictly passive systems (\ref{env_BC1}) and (\ref{env_BC2}), respectively. Obviously $H(y)\ge 0$ for all $y\in\HH$. Note that $H(y)$ is exactly the sum of the storage functions of the lossless Euler-Bernoulli beam \eqref{beam_PDE} - \eqref{beam_BC2}, the nonlinear spring-damper systems \eqref{env_BC_SD} and the strictly passive nonlinear dynamic feedback systems \eqref{env_BC1} and \eqref{env_BC2}, cf.~Figure~\ref{fb_inter}. In the following, it will be shown that $H$ qualifies as a Lyapunov function for the system \eqref{ivp}.

\begin{lem}\label{le_81}
 The function $H$ is continuous in $\HH$.
\end{lem}

\begin{proof}
The continuity of the terms in $H$ except for the $k_j$-terms is immediate. Due to the continuous embedding $H^2(0,L)\inj C^1([0,L])$ the continuity of the remaining $k_j$-terms follows as well.
\end{proof}

\begin{lem}\label{lem83}
Due to assumption \eqref{v:c_2} we have for any sequence $\{y_n\}_{n\in\N}\subset\HH$:
\[\sup_{n\in\N}H(y_n)<\infty\quad\aq\quad\sup_{n\in\N}\|y_n\|_{\HH}<\infty.\]
\end{lem}

\begin{proof}
 It suffices to notice that $\{V_j(z_{j,n})\}_{n \in \N}$ is unbounded iff $\{z_{j,n}\}_{n \in \N}$ is unbounded.
\end{proof}

We now define the generalized time derivative along the mild solution $y(t)$ of \eqref{ivp}, i.e.~for any initial value $y_0\in\HH$:
\[\dot H(y_0):=\limsup_{t\searrow0}\frac{H(y(t))-H(y_0)}t,\]
which may take the value $-\infty$.

\begin{lem}\label{lem82}
 For $y_0\in D(\LL)$ we have $\dot H (y_0) = \frac{\d^{+}}{\d t} H(y(t))|_{t=0} \le 0$, i.e.~$H$ is non-increasing along classical solutions. Here, $\frac{\d^+}{\d t}$ denotes the right derivative.
\end{lem}

\begin{proof}
 For $y_0\in D(\LL)$ the corresponding solution $y(t)$ of \eqref{ivp} is classical, and therefore has a continuous right derivative on $[0,T_{\mathrm{max}}(y_0))$. So we can directly compute
\begin{align*}
 \dot H(y_0)&=\frac{\mathrm{d}^+}{\mathrm{d}t}H(y(t))\big|_{t=0}\\
      &=a_1(z_1)\cdot\nabla V_1(z_1)+a_2(z_2)\cdot\nabla V_2(z_2)-d_1(v'(L))v'(L)-d_2(v(L))v(L).
\end{align*}
Thereby we have used \eqref{c:c_3}. The non-positivity of the generalized time derivative of the storage function $H$ can be directly concluded from \eqref{a:c_3} and \eqref{d:c}. Clearly, this is also a consequence of the passivity property of the feedback interconnected system according to Figure~\ref{fb_inter}. This concludes the proof.
\end{proof}

\begin{cor}
 For $y_0\in D(\LL)$ the corresponding classical solution $y(t)$ of \eqref{ivp} is global, i.e.~it exists for all  $t\in [0,\infty)$.
\end{cor}

\begin{proof}
 According to Lemma~\ref{lem82}, $H$ is non-increasing along $y(t)$. Thus, according to Lemma~\ref{lem83} no blowup occurs in $y(t)$, and we have according to Proposition~\ref{pro_61} that $T_{\mathrm{max}}(y_0)=\infty$.
\end{proof}

Since $\mathcal N$ is locally Lipschitz continuous and $D(\LL)\subset\HH$ is dense, we can apply Proposition~4.3.7 (ii) in \cite{Cazenave:Haraux} for the approximation of mild (non-classical) solutions:

\begin{prop}\label{prp_2}
 Let $y_0\in \HH$ and $\{y_{0,n}\}_{n\in\N}\subset D(\LL)$  be such that $y_{0,n}\to y_0$ in $\HH$. Denote by $y_n(t)$ the global classical solution of \eqref{ivp} to the initial value $y_{0,n}$ and by $y(t)$ the mild solution corresponding to the initial value $y_0$. Then $y_n(t)\to y(t)$ in $C([0,T];\HH)$ for any $T \in (0,T_{\mathrm{max}}(y_0))$.
\end{prop}


\begin{trm}\label{trm_1}
 For any $y_0\in\HH$ the corresponding solution $y(t)$ of the initial value problem \eqref{ivp} is global in time. Furthermore, $t\mapsto H(y(t))$ is non-increasing on $\R^+$ and $y(t)$ is uniformly bounded in $\HH$ on $[0, \infty)$.
\end{trm}

\begin{proof}
Consider $y_0\in \HH$ and a sequence $\{y_{0,n}\}\subset D(\LL)$ with $y_{0,n}\to y_0$ in $\HH$. Due to the convergence $y_n(t)\to y(t)$ for all  $ t \in [0,T_{\mathrm{max}}(y_0))$ shown in Proposition~\ref{prp_2} and the continuity of $H$, we get $H(y_n(t))\to H(y(t))$ for all $0\le t < T_{\mathrm{max}}(y_0)$. Since $H$ is non-increasing along every $y_n(t)$, this implies also that $t\mapsto H(y(t))$ is non-increasing on $[0,T_{\mathrm{max}}(y_0))$.
Thus, according to Lemma~\ref{lem83} no blowup of $y(t)$ can occur at $t = T_{\mathrm{max}}(y_0)$. So, according to Proposition~\ref{pro_61} the solution is global in time. Uniform boundedness of $y(t)$ now follows from Lemma~\ref{lem83}.
\end{proof}

\begin{cor}
 The function $H$ is a Lyapunov function for the initial value problem \eqref{ivp}.
\end{cor}

\begin{rem}\label{rem48}
Since all mild solutions are global, Proposition~\ref{prp_2} holds for any $T \in (0,\infty)$.
\end{rem}

For every $y_0\in\HH$ and the corresponding mild solution $y(t)$ we define $S(t)y_0:=y(t)$ for all $t\ge 0$. The family $S\equiv (S(t))_{t\ge 0}$ is a strongly continuous semigroup of nonlinear (bounded, continuous) operators in $\HH$, cf.~Theorem~9.3.2 in \cite{Cazenave:Haraux}.

\begin{rem}\label{mark}
Since \eqref{v:c_2} is only needed to show that no blowup of the solution occurs, we may replace it by the weaker assumption
\[\lim_{|z_j|\to\infty}V_j(z_j)>H(y_0)\tag{\ref{v:c_2}'},\]
depending on the initial condition $y_0$ for the problem \eqref{ivp}. Thereby we argue as follows: According to Theorem~\ref{trm_1} the function $t\mapsto H(y(t))$ is non-increasing (this is independent of \eqref{v:c_2}), which ensures that no blowup can occur in any component of $y(t)$ except for $z_j$. If now $z_1(t)$ or $z_2(t)$ would blowup, we would get $\lim_{t\to\infty}H(y(t))>H(y_0)$ according to (\ref{v:c_2}'). So $H(y(t))$ could not be non-increasing, which is a contradiction. So (\ref{v:c_2}') is sufficient to show that no blowup occurs and that the solution is global in time.
\end{rem}

\section{$\omega$-limit Set}\label{S5}

In the following, $S$ is the strongly continuous (nonlinear) semigroup generated by $\mathcal{A}$ on $\HH$, defined at the end of the previous section. In this section, we investigate possible $\omega$-limit sets of $S$. However, non-emptiness of the $\omega$-limit sets will only be discussed in the subsequent sections. For $y_0\in\HH$ we define the trajectory $\gamma(y_0)$ by
\[\gamma(y_0):=\bigcup_{t\ge 0}S(t)y_0.\]

\begin{de}[$\omega$-limit set]
Given the semigroup $S$, the $\omega$-limit set for $y_0\in \HH$ is denoted by $\omega(y_0)$, and is the following set:
\begin{align*}
\omega(y_0):=\{y\in \HH&: \exists \{t_n\}_{n\in\N}\subset\R^+, \lim_{n\to\infty}t_n =+\infty\wedge \, \lim_{n\to\infty}S(t_n)y_0 = y\}.
\end{align*}
It is possible that $\omega(y_0)=\emptyset$.
\end{de}

According to Proposition~9.1.7 in \cite{Cazenave:Haraux} we have:
\begin{lem}\label{lem72}
For $y_0\in\HH$ the set $\omega(y_0)$ is $S$-invariant, i.e.~$S(t)\omega(y_0)\subseteq \omega(y_0)$ for all $t\ge 0$.
\end{lem}


Let us consider now some fixed $y_0\in\HH$. According to the results of Section \ref{secc6}, the function $t\mapsto H(S(t)y_0)$ is non-increasing, and bounded from below by $0$. Therefore, the following limit exists:
\begin{equation}\label{frak_h}
 \mathfrak H(y_0):=\lim_{t\to\infty} H(S(t)y_0)\ge 0.
\end{equation}

\begin{lem}\label{lem73}
 Suppose $\omega(y_0)\neq \emptyset$. Then there holds
\[\forall y\in\omega(y_0):\quad H(y)=\mathfrak H(y_0).\]
In particular, $\dot H(y)=0$ for all $y\in\omega(y_0)$.
\end{lem}

\begin{proof}
For every $y\in\omega(y_0)$ there exists a sequence $\{t_n\}\subset\R^+$ such that $S(t_n)y_0\to y$. Since $H$ is continuous, cf.~Section \ref{secc6}, this implies that $H(y)=\lim_{n\to\infty}H(S(t_n)y_0)$. Due to \eqref{frak_h} the right hand side equals $\mathfrak H(y_0)$, and the result follows.
\end{proof}

We can use this lemma to identify the possible $\omega$-limit sets by investigating trajectories along which the Lyapunov function $H$ is constant.
\begin{lem}\label{srddx}
Let $y\in \HH$ such that $H(S(t)y)=\mathfrak H(y)$ for all $t\ge0$, i.e.~$H$ is constant along $\gamma(y)$. Then $\gamma(y)\subset \{y\in\HH:y=[u,v,0,0,0,0]^\top\}$.
\end{lem}

\begin{proof}
First, let $y\in D(\LL)$. We know from Lemma~\ref{lem82} and the corresponding proof that for all $t\ge0$
\begin{align}
 \frac{\d}{\d t} H(S(t)y)&= a_1(z_1)\cdot\nabla V_1(z_1)+a_2(z_2)\cdot\nabla V_2(z_2)\label{hdot}\\
 &\quad -d_1(v'(L))v'(L)-d_2(v(L))v(L),\nonumber
\end{align}
where we omitted the dependence on $t$ on the right hand side, i.e.~$[u, v, z_1,z_2,Jv'(L),Mv(L)]^\top\equiv S(t)y$. Now \eqref{hdot} is required to be zero, and according to \eqref{a:c_3} and \eqref{d:c} this holds iff $\xi=\psi=z_1=z_2=0$.

Now let $y\in \HH\setminus D(\LL)$. Then there is a sequence $\{y_n\}_{n \in \N}\subset D(\LL)$ such that $y_n\to y$ as $n\to \infty$. According to Remark~\ref{rem48} we have $S(t)y_n\to S(t)y$ uniformly on $[0, T]$ for any  $T>0$. Therefore, we have also for the components
\begin{align}
 z_{j,n}(t)\to  z_j(t),&\quad \text{ in }C([0, T];\R^{n_j}),\label{sd_1}\\
 Mv_n(t,L)\to  \psi(t),&\quad \text{ in }C([0, T];\R),\label{sd_2}\\
 Jv_n'(t,L)\to  \xi(t),&\quad \text{ in }C([0, T];\R).\label{sd_3}
\end{align}
Together with \eqref{hdot} this implies
\[\Big\{\frac{\d}{\d t} H(S(t)y_n)\Big\}_{n\in\N}\]
is a Cauchy sequence in $C([0,T];\R)$. Since $H$ is locally Lipschitz continuous in $\HH$, we also have that $\{H(S(t)y_n)\}_{n\in\N}$ is a Cauchy sequence in $C([0,T];\R)$, so altogether $\{H(S(t)y_n)\}_{n\in\N}$ is a Cauchy sequence in $C^1([0,T];\R)$. So  there exists a unique $h(t)\in C^1([0,T];\R)$ such that
\begin{equation}\label{daf2}
 H(S(t)y_n)\to h(t)\quad\text{in }C^1([0,T];\R).
\end{equation}
On the other hand we know that $\lim_{n\to\infty}H(S(t)y_n)= H(S(t)y)=\mathfrak H(y)$ for every $t\ge 0$, and hence $h(t)\equiv \mathfrak H(y)$. Together with \eqref{daf2} this implies $\frac{\d}{\d t} H(S(t)y_n)\to 0$ uniformly on $[0,T]$. By using \eqref{hdot} for every $y_n$ this now yields that in \eqref{sd_1} - \eqref{sd_3} we obtain the limits $z_j(t)=\xi(t)=\psi(t)=0$. So $S(t)y$ has to be of the form $S(t)y=[u(t),v(t),0,0,0,0]^\top$.
\end{proof}

Before we prove that the $\omega$-limit set consists only of the zero solution, we need the following technical lemma:{
\begin{lem} \label{aux}
Let $S$ be the nonlinear semigroup generated by $\mathcal A$. For every $ y_0 \in \HH$ and for all $t>0$ there holds:
\begin{equation} \label{A1}
        \int_0^t{S(s) y_0 \d s} \in D({\LL}),
\end{equation}
and
\begin{equation} \label{A2}
         S(t)y_0 - y_0={\A}\int_0^t S(s)y_0\d s + \int_0^t \mathcal N S(s)y_0\d s.
\end{equation}
\end{lem}
For the proof we only need the fact that $\mathcal L$ generates a $C_0$-semigroup, and that $\mathcal N$ is differentiable and locally Lipschitz continuous. Hence, the above result still holds true for more general operators $\mathcal L$ and $\mathcal N$, which satisfy the mentioned properties. The proof of Lemma~\ref{aux} is analogous to the proof of Lemma~5.4 in \cite{MSA}, see also \cite{thesis_DS} for a general version of this lemma.}

\begin{trm}\label{om_lise_0}
 Let $\emptyset\neq\Omega\subset\HH$ be an $S$-invariant set such that $H|_\Omega$ is constant. Then $\Omega=\{0\}$. In particular, for any $y_0\in\HH$ either $\omega(y_0)=\emptyset$ or $\omega(y_0)=\{0\}$.
\end{trm}

\begin{proof}
Take a fixed $y_0\in\Omega$, and let $y(t)$ be the corresponding mild solution of \eqref{ivp}. Clearly, $\gamma(y_0)\subset \Omega$, and according to Lemma~\ref{srddx} $y(t)$ is of the form $y(t)=[u(t),v(t),0,0,0,0]^\top$.

\medskip

\underline{Step 1  (linear system for $u(t)$, $v(t)$):} First we note that, according to \eqref{A1}, there holds for all $t\ge 0$:
\begin{align*}
0&=\int_0^t\psi(s)\d s=M\int_0^t v(s,L)\d s= M(u(t,L)-u_0(L)),\\
0&=\int_0^t\xi(s)\d s=J\Big(\int_0^t v(s,x)\d s\Big)'\Big|_{x=L}\\
&= J(u'(t,L)-u_0'(L)).
\end{align*}
Thus $u(t,L)$ and $u'(t,L)$ are constant in time. According to \eqref{A2} the (projected) mild solution $y_p(t)=[u(t),v(t)]^\top$ satisfies the following system (i.e.~the first, second, fifth, and sixth component of \eqref{A2}):
\begin{samepage}
\begin{subequations}\label{nl_ap_system}
\begin{align}
u(t)- u_0 &= \int_0^t v(s,x)\d s,\\
v(t)- v_0 &= -\frac\Lambda\rho \Big(\int_0^t u(s,x)\d s\Big)^\rom{4}\\
0 &= \Lambda\Big(\int_0^t u(s,x)\d s\Big)^{\prime\prime}\Big|_{x=L}\label{5.9c}\\
&\quad+K_1\Big(\int_0^t\!\! u(s,x)\d s\Big)'\Big|_{x=L}\!\!+\int_0^t\!\!\kappa_1(u'(s,L))\d s,\nonumber\\
0 &= -\Lambda\Big(\int_0^t u(s,x)\d s\Big)^{\prime\prime\prime}\Big|_{x=L}\\
&\quad+K_2\Big(\int_0^t \!\!u(s,x)\d s\Big)\Big|_{x=L}+\int_0^t\!\!\kappa_2(u(s,L))\d s.\nonumber
\end{align}
\end{subequations}
\end{samepage}
\noindent Mild solutions satisfy $u\in C(\R^+;\tilde H_0^2(0,L))$. Hence, we can interchange the integration and differentiation in the last term of \eqref{5.9c}. Using the fact that $u^{\prime}(t,L)$ is constant, we have (for $u_0'(L)\neq 0$):
\begin{align*}
\int_0^t \kappa_1(u'(s,L))\d s& = t \kappa_1(u_0'(L))\\
&= \frac{\kappa_1(u_0'(L))}{u_0'(L)}\Big(\int_0^t u(s,x)\d s\Big)'\Big|_{x=L}.
\end{align*}
Next we define the constants (since $\kappa_j(0)=0$):
\begin{align}
\begin{split}\label{5_star}
\tilde K_1&:=K_1+\frac{\kappa_1(u_0'(L))}{u_0'(L)},\quad \text{if }u_0'(L)\neq 0, \text{ else }\tilde K_1:=K_1 ,\\
\tilde K_2&:=K_2+\frac{\kappa_2(u_0(L))}{u_0(L)},\quad \text{if }u_0(L)\neq 0,\text{ else }\tilde K_2:=K_2.
\end{split}
\end{align}
With this we may rewrite \eqref{nl_ap_system} as
\begin{subequations}\label{int_verss}
\begin{align}
u(t)&- u_0 = \int_0^t v(s)\d s,\\
v(t)&- v_0 = -\frac\Lambda\rho \Big(\int_0^t u(s)\d s\Big)^\rom{4}\label{410b}\\
0&= \Lambda\Big(\int_0^t\!\! u(s,x)\d s\Big)^{\prime\prime}\Big|_{x=L}\!\!\!
+\tilde K_1\Big(\int_0^t \!\!u(s,x)\d s\Big)'\Big|_{x=L},\label{5.10c}\\
0&= -\Lambda\Big(\int_0^t \!\!u(s,x)\d s\Big)^{\prime\prime\prime}\Big|_{x=L}\!\!\!
+\tilde K_2\int_0^t \!\!u(s,x)\d s\Big|_{x=L},\label{5.10d}
\end{align}
\end{subequations}
making this system linear. Thus, the projected vector $y_p(t)=[u(t),v(t)]^\top$ is the unique mild solution of
\begin{subequations}\label{ivp_p}
 \begin{align}
   (y_p)_t &= {\A}_p y_p,\label{ivp_p_1}\\
  y_p(0)&= [u_0,v_0]^\top,
 \end{align}
\end{subequations}
with the operator
\[{\A}_p:\begin{bmatrix}
u\\v
\end{bmatrix} \mapsto\begin{bmatrix}
v\\
-\frac\Lambda\rho u^\rom{4}
\end{bmatrix}.\]
The equations \eqref{5.10c} and \eqref{5.10d} are incorporated into the domain $D({\A}_p)$. For further details on the operator ${\A}_p$ in the space $\HH_p$ see the Appendix.
\medskip

\underline{Step 2 (proof of $u(t,L)=u'(t,L)=0$):}
We now investigate solutions of the projected problem \eqref{ivp_p} with the additional property that $u(t,L)$ and $u'(t,L)$ are constant in time. Since the semigroup $\e^{t{\A}_p}$ is unitary in $\HH_p$, we know in particular that $\|v(t)\|_{L^2}\le C= \frac{1}{\sqrt{\rho}} \|y_p(0)\|_{\HH_p}$ for all $t\ge 0$ (cf.~\eqref{App1}). Applying the norm to \eqref{410b} this yields\begin{equation}\label{dot_h_4}
\sup_{t \ge 0}{\Big\|\Big( \int_0^t{u(s) \d{} s}\Big)^{\rom{4}}\Big\|_{L^2(0,L)}} < \infty.
\end{equation} Next we apply the following Gagliardo-Nirenberg inequalities (cf.~\cite{nirenberg}), which guarantee the existence of a $C>0$ such that there holds for all $t\ge 0$:
\begin{equation}\label{gag_nir}
\begin{split}
\Big\|\int_0^t u(s) \d s\Big\|_{L^\infty(0,L)} & \le C \Big\|\Big( \int_0^t u(s) \d{} s \Big)^{\rom{4}}\Big\|_{L^2(0,L)}^{\frac 18}
\cdot\Big\|\int_0^t u(s) \d{} s\Big\|_{L^2(0,L)}^{\frac 78}, \\
\Big\|\int_0^t u'(s) \d s\Big\|_{L^\infty(0,L)}& \le C \Big\|\Big( \int_0^t u(s) \d{} s \Big)^{\rom{4}}\Big\|_{L^2(0,L)}^{\frac 38}
\cdot\Big\|\int_0^t u(s) \d{} s\Big\|_{L^2(0,L)}^{\frac 58}.
\end{split}
\end{equation}
The first factor on the right hand side in both inequalities is uniformly bounded (with respect to $t$) due to \eqref{dot_h_4}. For the second factor we observe that, according to Theorem~\ref{trm_1}, $t\mapsto\|u(t)\|_{L^2(0,L)}$ is uniformly bounded, and therefore $t\mapsto \|\int_0^t u(s)\d s\|_{L^2(0,L)}$ grows at most linearly. Hence, \eqref{gag_nir} implies that $t\mapsto\|\int_0^t u(s,L) \d s\|_{L^\infty(0,L)}$ grows at most like $t^{\frac 78}$ and $t\mapsto\|\int_0^t u'(s,L) \d s\|_{L^\infty(0,L)}$ at most like $t^{\frac 58}$ as $t\to\infty$. But this contradicts the fact that $u(t,L)$ and $u'(t,L)$ are constant, unless $u_0(L)= u_0'(L)=0$. This shows that $u(t,L)=u'(t,L)=0$ for all $t\ge 0$.

\medskip

\underline{Step 3 (Holmgren's Theorem):}
By iterated $t$-integration we shall now construct $C^4$-solutions of \eqref{ivp_p_1}, for which we can apply the Holmgren Uniqueness Theorem~\cite[Section 3.5]{john}. So we define $y_1(t)\equiv[u_1(t),v_1(t)]^\top:=\int_0^t y_p(s)\d s+{\A}_p^{-1}[u_0,v_0]^\top$. Due to Theorem~1.2.4 in \cite{pazy} and Lemma~\ref{invertdf} we have $y_1(t)\in D({\A}_p)$ for all $t\ge 0$. So $y_1$ is a classical solution of \eqref{ivp_p_1} to the initial condition $y_1(0)={\A}_p^{-1}[u_0,v_0]^\top$. Furthermore, because of $u(t,L)=u'(t,L)=0$, again $u_1(t,L),u_1'(t,L)$ are constant in time. Completely analogous to the previous step we can show again that $u_1(t,L)=u_1'(t,L)=0$.

Next we shall construct solutions of higher regularity. We iterate the previous step and define recursively  $y_n(t)\equiv[u_n(t),v_n(t)]^\top:=\int_0^t y_{n-1}(s)\d s+{\A}_p^{-n}[u_0,v_0]^\top$, which solves \eqref{ivp_p_1} classically with the initial condition $y_n(0)={\A}_p^{-n}[u_0,v_0]^\top$. Again we have  $u_n(t,L)=u_n'(t,L)=0$. Furthermore, by definition we have on the one hand ${\A}_py_n(t)=y_{n-1}(t)$. And on the other hand ${\A}_p[u_n,v_n]^\top=[v_n,-\Lambda/\rho \,u_n^\rom{4}]^\top$, so we can show inductively that $y_n\in C(\R^+; \tilde H_0^{2n+2}(0,L)\times \tilde H_0^{2n}(0,L))$. Now we consider the solution $u_n$ for $n\ge 2$. It satisfies the following partial differential equation with boundary conditions:
\begin{subequations}\label{clas}
\begin{align}
 (u_n)_{tt} & = -\frac\Lambda\rho u_n^{\rom{4}},\label{clasa}\\
[u_n(0,x),(u_n)_t(0,x)]^\top&={\A}_p^{-n}[u_0,v_0]^\top,\\
 u_n(t,0)&=u_n'(t,0)=0,\\
 u_n(t,L)&=\ldots=u_n^{\prime\prime\prime}(t,L)=0.\label{ctf}
\end{align}
\end{subequations}
By using \eqref{clasa}, $u_n\in  C(\R^+;\tilde H_0^{2n+2}(0,L))$, and the fact that $(u_n)_t=v_n\in C(\R^+;\tilde H_0^{2n}(0,L))$, we obtain the following properties for the mixed fourth order space-time derivatives of $u_n$:
\begin{align*}
 u_n^{\rom{4}} &\in C(\R^+;\tilde H_0^{2n-2}(0,L)),\\
 (u_n)_t^{\prime\prime\prime}&\in C(\R^+;\tilde H_0^{2n-3}(0,L)),\\
 ({u}_n)_{tt}^{\prime\prime}&=-\frac\Lambda\rho u_n^{\rom{6}} \in C(\R^+;\tilde H_0^{2n-4}(0,L)),\\
 (u_n)_{ttt}'&=- \frac\Lambda\rho v_n^{\rom{5}}\in C(\R^+;\tilde H_0^{2n-5}(0,L)),\\
(u_n)_{tttt}&=\frac{\Lambda^2}{\rho^2} u_n^{\rom{8}}\in C(\R^+;\tilde H_0^{2n-6}(0,L)).
\end{align*}
So for $n\ge4$, all mixed derivatives of $u_n$ of order four lie in $C(\R^+;\tilde H_0^2(0,L))\subset C(\R^+\times [0,L])$. Thus $u_n(t,x)$ is a $C^4$-solution of \eqref{clas}.

Now we can apply the Holmgren Uniqueness Theorem~\cite[Section 3.5]{john} on the strip $\R^+\times (0,L)$. Due to \eqref{ctf} all partial derivatives up to order $3$ of $u_4$ vanish on the line $\R^+\times \{L\}$. Therefore, Holmgren's Uniqueness Theorem implies that $u_4=0$ has to hold everywhere in this strip. (See also the proof of Lemma~3 in \cite{littman} for a similar result -- but without a detailed proof.) Therefore ${\A}_p^{-4}[u_0,v_0]^\top=0$ has to hold, and since ${\A}_p^{-1}$ is injective, this yields $[u_0,v_0]^\top=0$. Since $y_p(t)=\e^{t{\A}_p}[u_0,v_0]^\top$, we conclude that $u(t)=v(t)=0$ for all $t\ge 0$, and hence $\Omega=\{0\}$.

{For the final statement of the theorem, let $\omega(y_0) \ne \emptyset$. Then, by  Lemma~\ref{lem72} $\omega(y_0)$ is $S$-invariant, and by Lemma~\ref{lem73} $H$ is constant on $\omega(y_0)$. Hence, by the first statement of Theorem~\ref{om_lise_0}, $\omega(y_0) = \{0\}$.}
\end{proof}

As a consequence we obtain convergence to zero for trajectories with $\omega(y_0)\neq\emptyset$:

\begin{cor}\label{i_r}
 If $\omega(y_0)\neq\emptyset$ for some $y_0\in\HH$, then
 \[\lim_{t\to\infty}\|S(t)y_0\|_{\HH}=0.\]
\end{cor}

\begin{proof}\sloppy
If $\omega(y_0)\neq\emptyset$ then there exists a sequence $\{t_n\}_{n\in\N}$ with $t_n\to\infty$ such that ${\lim_{n\to\infty}S(t_n)y_0=0}$. Due to the continuity of the Lyapunov function $H$ this implies that
\[\lim_{n\to\infty}H(S(t_n)y_0)=0.\]
But since $t\mapsto H(S(t)y_0)$ is non-increasing, this implies that even
\[\lim_{t\to\infty}H(S(t)y_0)=0.\]
Due to the continuity of $H$ this implies that $\|S(t)y_0\|_{\HH}\to 0$ as $t\to \infty$.
\end{proof}


\section{Asymptotic Stability -- Linear $k_j$}\label{S6}

In the case where the $k_j$ are linear we are able to show precompactness for all trajectories, even for the mild, non-classical solutions. This will yield that the $\omega$-limit set $\omega(y_0)$ is always non-empty, and hence the asymptotic stability of the nonlinear semigroup will follow.

\begin{lem}\label{pro21}
 Let $y_0\in \HH$, and $y(t)$ be the corresponding mild solution of \eqref{ivp}. For $j=1,2$ let $\kappa_j=0$. Then $\mathcal Ny(t)\in L^1(\R^+;\HH)$.
\end{lem}

\begin{proof}
 First, let us assume that $y_0\in D(\LL)$, so $y(t)$ is a classical solution. We know from Theorem~\ref{trm_1} that $H(y(t))$ is non-increasing. By integrating \eqref{hdot} with respect to time we obtain
\begin{equation}\label{int_hdot}
\begin{split}
H(y(T))-H(y_0)&=\int_0^T\!\Big[ -d_1\!\left(\frac{\xi}{J}\right)\frac{\xi}{J}-d_2\!\left(\frac{\psi}{M}\right)\frac{\psi}{M}\\
&\qquad+a_1(z_1)\cdot\nabla V_1(z_1)+a_2(z_2)\cdot\nabla V_2(z_2)\Big]\d t=:I_T(y_0),
\end{split}
\end{equation}
where all terms on the right hand side include elements of the vector $y(t)$, thus depend on $t$. If we let $T\to\infty$, we know that $H(y(T))\to\mathfrak H(y_0)$, i.e.~the limit exists and the integral $I_\infty(y_0)$ is finite.

Now we consider $y_0\in \HH$, and $y(t)$ is the corresponding mild solution of \eqref{ivp}. Let $\{y_{0,n}\}_{n\in\N}\subset D(\LL)$ be a sequence with $y_{0,n}\to y_0$. According to Proposition~\ref{prp_2} and Remark~\ref{rem48} the corresponding classical solutions $y_n(t)$ converge to $y(t)$ in $C([0,T];\HH)$ for all $T>0$. Therefore $I_T(y_{0,n})\to I_T(y_0)$, cf.~\eqref{int_hdot}. Due to continuity of $H$, also $H(y_n(T))-H(y_{0,n})\to H(y(T))-H(y_0)$ as $n\to\infty$. Thus, \eqref{int_hdot} also holds for mild solutions for any $T>0$. Since $H(y(T))\to \mathfrak H(y_0)\in[0,H(y_0)]$ as $T\to\infty$, the integral $I_\infty(y_0)$ is finite.

Now we know that for any (mild) solution $y(t)$ the integral $I_\infty(y_0)$ from \eqref{int_hdot} is finite. Since all the terms in the integrand of \eqref{int_hdot} are non-positive, we conclude together with \eqref{a:c_4} and \eqref{d:c}  that
\begin{equation}\label{el2}
  z_j(t),\psi(t),\xi(t)\in L^2(\R^+).
\end{equation}
Under the assumptions we made in Section \ref{secc5} for the functions occurring in the nonlinear operator $\mathcal N$, the properties \eqref{el2} immediately imply $\mathcal Ny(t)\in L^1(\R^+;\HH)$.
\end{proof}

\begin{rem}
To obtain $\mathcal Ny(t)\in L^1(\R^+;\HH)$ in the above proof, we used in \eqref{int_hdot} that the nonlinear damping functions $d_j$ include a non-vanishing linear part (i.e. $D_j > 0$). The same assumption will also be needed in Step 3 of the proof of Lemma~\ref{precomp_2} below. However, in the nonlinear spring-damper system of \cite{MSA}, a locally quadratic growth of the damper law was sufficient. From a practical point of view, this is not restrictive at all.
\end{rem}

We note that \eqref{int_hdot} does not give any control on $u(t,L)$ and $u'(t,L)$ (in the sense of \eqref{el2}). Hence, the linearity assumption $\kappa_j=0$ was crucial for the above proof.

\begin{trm}\label{fin_trm}
{Let $\kappa_j=0$ for $j=1,2$.} For any $y_0\in \HH$ there holds $\lim_{t\to\infty} S(t)y_0=0$, i.e.~the semigroup $S$ is asymptotically stable.
\end{trm}

\begin{proof}
{Our aim is to apply a version of Theorem~4 in \cite{daf_sle}. It states that if $\mathcal L$ is a linear, maximal dissipative operator with $(\lambda-\mathcal L)^{-1}$ is compact for some $\lambda>0$, and $f\in L^1(\R^+;\HH)$, then every mild solution of the Cauchy problem $\dot y(t) = \mathcal Ly(t) +f(t)$ has a precompact trajectory.}

According to Remark~\ref{hypoer} the linear part ${\A}$ of $\LL$ is a maximal dissipative operator on $\HH$. As seen in the proof of Theorem~\ref{teo73}, ${\A}^{-1}$ exists and is compact. Since ${\A}$ generates a $C_0$-semigroup of contractions,  $(\lambda-{\A})^{-1}$ exists and is compact for all $\lambda>0$. {Finally, according to Lemma~\ref{pro21} we know that $\mathcal N y(t)\in L^1(\R^+;\HH)$ for $y(t):=S(t) y_0$.} Due to these facts, we can apply Theorem~4 in \cite{daf_sle} with $f(t):=\mathcal N y(t)$. This shows that the $\omega$-limit set $\omega(y_0)$ is non-empty. Thus, due to Corollary~\ref{i_r} and Theorem~\ref{om_lise_0}, we conclude $\omega(y_0)=\{0\}$ and that the entire solution $y(t)$ converges to zero.
\end{proof}



\section{Asymptotic Stability -- Nonlinear $k_j$}\label{S7}

According to Corollary~\ref{i_r}, any trajectory with a non-empty $\omega$-limit set already is asymptotically stable.
Thus, in order to complete the discussion we show in this section that (at least) any classical trajectory possesses a non-empty $\omega$-limit.
 We do this by proving that every classical trajectory is precompact. To this end we follow a strategy introduced in \cite{MSA}.
We begin with the following preparatory result (which would be obvious for linear semigroups):
\begin{lem} \label{regularity_2}
Let $y(t)$ be a (mild) solution of \eqref{ivp} and let $y_0 \in D({\LL}^2):=\{y\in D({\LL}): \LL y\in D(\LL)\}$. Then $y\in C^2([0,\infty); \mathcal{H})$ and $y_t(t) \in D(\LL)$ for all $t>0$.
\end{lem}

\begin{proof}
If we already knew that $y\in C^2([0,\infty);\HH)$, it would follow that $\tilde y:= y_t$ satisfies
\begin{equation}\label{ivp_2}
\tilde y_t={\A} \tilde y+\begin{bmatrix}
0\\0\\ \alpha_1'(z_1) \tilde{z_1} + \frac1J [\beta_1'(z_1)\tilde{z}_1 \xi + \beta_1(z_1) \tilde{\xi} ]\\ \alpha_2'(z_2) \tilde{z}_2 + \frac1M [\beta_2'(z_2)\tilde{z}_2 \psi + \beta_2(z_2) \tilde{\psi} ]\\[0.5mm]
- \gamma_1'(z_1)\tilde{z}_1 - \frac1J \delta_1'\big(\frac{\xi}{J}\big) \tilde{\xi}- \kappa_1'(u'(L)) \tilde{u}'(L) \\[0.5mm] - \gamma_2'(z_2)\tilde{z}_2 - \frac1M \delta_2'\big( \frac{\psi}{M}\big) \tilde{\psi} - \kappa_2'(u(L)) \tilde{u}(L)
\end{bmatrix}.
\end{equation}
However, at this point we only know that $y(t)\in C^1([0,\infty);\HH)$, see Proposition~\ref{pro_61}. Motivated by \eqref{ivp_2} we  define the following functions for this fixed $y(t)=[u, v,z_1,z_2,\xi,\psi]^\top(t)$:
\begin{align*}
G_1(t,Z)& := \alpha_1'(z_1) \zeta_1 + \frac1J [\beta_1'(z_1)\zeta_1 \xi + \beta_1(z_1) \Xi ],\\
G_2(t,Z)& := \alpha_2'(z_2) \zeta_2 + \frac1M [\beta_2'(z_2)\zeta_2 \psi + \beta_2(z_2) \Psi ],\\
G_3(t,Z)& := - \gamma_1'(z_1)\zeta_1 - \frac1J \delta_1'\Big(\frac{\xi}{J}\Big) \Xi - \kappa_1'(u'(L)) U'(L),\\
G_4(t,Z)& := - \gamma_2'(z_2)\zeta_2 - \frac1M \delta_2'\Big(\frac{\psi}{M}\Big) \Psi - \kappa_2'(u(L)) U(L),\\
\end{align*}
where $Z:=[U,V,\zeta_1,\zeta_2, \Xi, \Psi]^\top\in\HH$. Since $y(t)$ is a classical solution, it follows {from the regularity assumptions of the coefficients that $t\mapsto G_j(t,Z)$ lies in $C^1$} for all $j=1,\dots,4$. As a consequence the operator $\tilde{\mathcal{N}} :[0,T]\times \HH \to \HH$ defined by
\[\tilde{\mathcal{N}} (t,Z):=[0,0, G_1(t,Z), G_2(t,Z), G_3(t,Z), G_4(t,Z)]^{\top},\] is {Lipschitz continuous for any fixed $T>0$, and linear in $Z\in\HH$}. Now the linear, non-autonomous initial value problem
\begin{subequations}\label{ivp_z}
\begin{align}
Z_t &= {\A}Z+\tilde{\mathcal{N}}(t,Z),\\
Z(0)&=Z_0 \in \HH,
\end{align}
\end{subequations}
is considered. According to Theorem~6.1.2 in \cite{pazy} there exists a unique global mild solution $Z(t)$ of \eqref{ivp_z} for every $Z_0\in \HH$. If $Z_0\in D({\A})$ this solution is classical, see Theorem~6.1.5 in \cite{pazy}.

Our next aim is to show that for the classical solution $y(t)$ fixed in the beginning, the (continuous) function $y_t(t)$ is indeed a mild solution of \eqref{ivp_z} for $Z_0=\LL y_0$: Since $y(t)$ satisfies the Duhamel formula \eqref{milde} and is differentiable, we obtain after differentiating with respect to $t$
\begin{equation}\label{diffmild}
y_t(t)= \e^{t{\A}}{\A} y_0+\frac{\d}{\d t}\int_0^t{\e^{(t-s){\A}}\mathcal{N} y(s)\d s}.
\end{equation}
According to the proof of Corollary~4.2.5 in \cite{pazy} the following statement holds true
\begin{equation} \label{aux_duh}
\frac{\d}{\d t}\!\!\int_0^t\!\!{\e^{(t-s){\A}}\mathcal{N} y(s)\d s} = \e^{t{\A}}\mathcal{N} y_0 + \!\!\int_0^{t}\!\! \e^{(t-s){\A}}\frac{\d}{\d s}\mathcal{N} y(s)\d s.
\end{equation}
Inserting \eqref{aux_duh} in \eqref{diffmild} yields that $y_t(t)$ fulfills the Duhamel formula for \eqref{ivp_z}. As a consequence $y_t(t)$ is the unique mild solution of \eqref{ivp_z} to the initial condition $Z_0=\LL y_0$. Moreover, we know that this mild solution $Z(t)=y_t(t)$ is a classical solution of $\eqref{ivp_z}$ if $\LL y_0\in D(\LL)$, i.e.~$y_0\in D(\LL^{2})$. Hence $y_t\in C^1(\R^{+} ; \HH)$ and $y\in C^2(\R^{+} ; \HH)$.
\end{proof}

\begin{lem} \label{precomp_2}
The trajectory $\gamma(y_0)$ is precompact in $\HH$ for $y_0 \in D({\LL}^2)$. Moreover, there exists a constant $C>0$ such that
\begin{equation} \label{unif_bound}
\| y_t(t)\|_{\HH} \le C, \qquad \forall t\ge0,
\end{equation}
where $C$ depends continuously on $\|y_0\|_{\HH}$ and $\|y_t(0)\|_{\HH}$.
\end{lem}

\begin{proof}
In order to prove precompactness of the trajectory, it suffices to show that
\[\sup_{t>0}\|\LL y(t)\|_\HH<\infty,\] due to the compact embeddings $H^4(0,L)\inj\inj H^2(0,L)\inj\inj L^2(0,L)$.
However, this is equivalent to showing that $y_t$ is uniformly bounded in $\HH$, since $y_t = \LL y$. Again, this is equivalent to
\begin{align*}
H(y_t) &=  \frac{\rho}2 \int_0^L u_{tt}^2 \d x + \frac\Lambda2 \int_0^L\big( u^{\prime\prime}_t\big)^2  \d x +  \frac J2  \big(u'_{tt}(L)\big) ^2+ \frac{M}{2}  \big(u_{tt}(L)\big)^2
\\
&\quad  + \int_0^{u_t'(L)}k_1(s)\d s+ \int_0^{u_t(L)}k_2(s)\d s +V_1((z_1)_t) + V_2((z_2)_t)
\end{align*}
being uniformly bounded. Since $y(t)$ is a classical solution, we have the following equalities
\[u_t(L)=\frac \psi M,\quad u_t'(L)=\frac\xi J.\]
According to Theorem~\ref{trm_1} those terms are always uniformly bounded.
Moreover, due to regularity of the functions $a_j, b_j$ and Theorem~\ref{trm_1} we see from \eqref{EBB_4} and \eqref{EBB_5} that $(z_j)_t\in L^\infty(\R^+)$ for $j=1,2$.
Therefore, the boundedness of $H(y_t)$ is equivalent to the boundedness of the functional
\begin{align*}
\tilde H(y_t)& := \frac{\rho}2 \int_0^L u_{tt}^2 \d x + \frac\Lambda2 \int_0^L\big( u^{\prime\prime}_t\big)^2  \d x 
+  \frac J2  \big(u'_{tt}(L) \big)^2
+ \frac{M}{2}   \big(u_{tt}(L)\big)^2.
\end{align*}
Hence, our aim is to derive a system of equations satisfied by $y_t(t)$, and then to show that $\tilde H(y_t)$ is uniformly bounded.

\medskip

\underline{Step 1 (Time derivative of the system)}: According to Lemma~\ref{regularity_2},  $y(t) \in C^2([0, \infty); \HH)$. Differentiating \eqref{beam_PDE} - \eqref{beam_BC2} with respect to time hence shows that $y_t$
is the classical solution of the following system
\begin{subequations}  \label{system_time_der}
\begin{align}
   \rho  u_{ttt}  + \Lambda u^{\rom{4}}_t & =  0, \\
 u_t (t,0) =  u'_t (t,0) & =  0, \label{derc}\\
 \Lambda u^{\prime\prime}_t (t, L)+ J   u'_{ttt} (t, L)  + (\tau_{e})_t(t)&  =  0, \label{derd}\\
- \Lambda u^{\prime\prime\prime}_t (t,L)+ M u_{ttt} (t, L) + (f_{e})_t(t) &  =  0,\label{dere}
\end{align}
\end{subequations}
where
\begin{align}
\begin{split}\label{theta_def}
\tau_e & :=  c_1(z_1) +d_1( u_t'(L)) + k_1(u'(L)),\\
f_e & :=  c_2(z_2) + d_2(u_t(L)) + k_2(u(L)).
\end{split}
\end{align}
Therefore, from \eqref{theta_def} it follows
\begin{align}
\begin{split}\label{theta_der}
(\tau_e)_t &\! =\! \nabla c_1(z_1)\!\cdot\!(z_1)_t \!+\! d_1'(u'_{t}(L)) u'_{tt}(L) \!+\! k_1'(u'(L)) u'_t(L),\\
(f_e)_t & \!= \!\nabla c_2(z_2)\!\cdot\! (z_2)_t\! +\! d_2'(u_{t}(L)) u_{tt}(L) \!+\! k_2'(u(L)) u_t(L),
\end{split}
\end{align}
and from \eqref{EBB_4} and \eqref{EBB_5}, we obtain
\begin{subequations}  \label{controller_time_der}
\begin{align}
 (z_1)_{tt} & =  [J_{a_1}(z_1) + u'_t(L) J_{b_1}(z_1) ](z_1)_t + b_1(z_1) u'_{tt}(L), \label{z1tt}\\
 (z_2)_{tt} & =  [J_{a_2}(z_2) + u_t(L) J_{b_2}(z_2) ](z_2)_t + b_2(z_2) u_{tt}(L), \label{z2tt}
\end{align}
\end{subequations}
where $J_{a_j}$, $J_{b_j}$ denote the Jacobian matrices of the functions $a_j$ and $b_j$, respectively.
Note that from Lemma~\ref{lem82} it follows that $z_j(.), u_t(.\, , L)=\frac\psi M, u'_t(.\, , L)=\frac\xi J \in L^2(\R^{+})$ (cf. \eqref{el2} for a similar conclusion). Therefore \eqref{EBB_4} and \eqref{EBB_5} imply $(z_j)_t \in L^2(\R^{+})$.

\medskip

\underline{Step 2 (Time derivative of $\tilde H(y_t)$)}:
We obtain
\begin{align}
\begin{split} \label{lyapunov_time_der}
 \frac{\d}{\d t}\tilde H(y_t) & =  \rho  \int_0^L u_{ttt} u_{tt} \d{}x+ \Lambda \int_0^L u^{\prime\prime}_{tt} u^{\prime\prime}_t  \d{}x \\
 &\quad +
J u'_{ttt}(L) {u}'_{tt}(L) + M u_{ttt}(L){u}_{tt}(L) \\
& =  {u}_{tt}(L) \big(M u_{ttt}(L)  - \Lambda u^{\prime\prime\prime}_t(L)\big)\\&\quad +  {u}'_{tt}(L) \big(  \Lambda u^{\prime\prime}_t(L)  + J u'_{ttt}(L)  \big)\\
& = - {u}_{tt}(L) \big( 
(z_2)_t ^{\top} \nabla c_2(z_2) +  k_2'(u(L)) u_t(L)\big) \\
&\quad  -  {u}'_{tt}(L) \big( 
(z_1)_t^{\top} \nabla c_1(z_1) + k_1'(u'(L)) u'_t(L)\big)\\
&\quad  - d_2'(u_{t}(L)) \big(u_{tt}(L)\big)^2 - d_1'(u'_{t}(L))\big( u'_{tt}(L)\big)^2,
\end{split}
\end{align}
\noindent{}where we have performed partial integration in $x$ twice, and then used \eqref{system_time_der} and \eqref{theta_der}. 
Integrating \eqref{lyapunov_time_der} on the time interval $[0,t]$, for some arbitrary $t \in \R^{+}$, we get with \eqref{d:c}
\begin{equation} \label{lyapunov_time_int}
\tilde H(y_t(t))  \le\tilde  H(y_t(0)) + I_1(t) + I_2(t),
\end{equation}
where
\begin{align*}
I_1(t) &:=  -\int_0^t \!\!{u}'_{tt}(L) \Big( 
 (z_1)_t^{\top} \nabla c_1(z_1) + k_1'(u'(L)) u'_t(L)\Big) \d{}s,\\
I_2(t) &:=  -\int_0^t \!\!{u}_{tt}(L) \Big( 
 (z_2)_t ^{\top} \nabla c_2(z_2) +  k_2'(u(L)) u_t(L)\Big) \d{}s.
\end{align*}

\medskip

\underline{Step 3 (Boundedness of $I_1$ and $I_2$)}: Next, we show uniform boundedness for each component of $I_2$ by using partial integration in time:
\begin{align*}
-\int_0^t u_{tt}(L) k_2^{\prime}(u(L)) u_t(L) \d{}s  
 &= -\frac12 \left(u_t(t,L)\right)^2 k_2^{\prime}(u_t(t,L)) + \frac12 \left(u_t(0,L)\right)^2 k_2^{\prime}(u_t(0,L))\\
&\quad+ \frac12 \int_0^t u_t(L)^3 k_2^{\prime\prime}(u(L)) \d{}s \le C, \quad \forall t\ge 0.
\end{align*}
Further, it holds that
\begin{align*}
\int_0^t{u}_{tt}(L) (z_2)_t^{\top} \nabla c_2(z_2)  \d s 
& = {u}_{t}(t,L) (z_2)_t(t)^{\top} \nabla c_2(z_2(t)) 
- {u}_{t}(0, L) (z_2)_t(0)^{\top} \nabla c_2(z_2(0))  \\
&\qquad- \int_0^t u_t(L)\big [ (z_2)_t^{\top} \big[\hess(c_2)(z_2)\big](z_2)_t
+  (z_2)_{tt}^{\top} \nabla c_2(z_2) \big]\d{}s.
\end{align*}
Since $c_2 \in C^2(\R^{n_2};\R)$ and $z_2(t)\in L^\infty(\R^+)$, it follows that \[ \int_0^t| u_t(L)(z_2)_t^{\top} \big[\hess(c_2)(z_2)\big](z_2)_t |\d s \le C \int_0^t {| (z_2)_t |^2 \d s},\]
and (with \eqref{controller_time_der})
\begin{align*}
\int_0^t u_{t}(L) (z_2)_{tt}^{\top} \nabla c_2(z_2)  \d{}s  
& =\int_0^t  u_{t}(L) [J_{a_2}(z_2) (z_2)_t + u_t(L) J_{b_2}(z_2) (z_2)_t]^{\top} \nabla c_2(z_2)  \d{}s\\
&\quad +  \int_0^t  \nabla c_2(z_2)^{\top} b_2(z_2) u_{tt}(L) u_{t}(L)\d{}s \\
 &=  \int_0^t u_{t}(L) [J_{a_2}(z_2) (z_2)_t + u_t(L) J_{b_2}(z_2) (z_2)_t]^{\top} \nabla c_2(z_2)  \d{}s\\
 &\quad+  \frac12 \nabla c_2(z_2(t))^{\top} b_2(z_2(t))  u_{t}(t,L)^2 \\
 &\quad- \frac12 \nabla c_2(z_2(0))^{\top} b_2(z_2(0))  u_{t}(0,L)^2 \\
 &\quad- \frac12 \int_0^t \Big( u_{t}(L)^2 (z_2)_t^{\top}\\
 &\qquad\qquad\cdot \big[J_{b_2}(z_2)^{\top} \nabla c_2(z_2) + \hess(c_2)(z_2)b_2(z_2)\big]\Big)\d{}s  \\
 & \le C \int_0^t |u_{t}(L)|^2 + |(z_2)_t |^2 \d{}s \\
 &\quad + \frac12 \nabla c_2(z_2(t))^{\top} b_2(z_2(t))  u_{t}(t,L)^2 \\
 &\quad- \frac12 \nabla c_2(z_2(0))^{\top} b_2(z_2(0))  u_{t}(0,L)^2.
\end{align*}
For the estimate of the second integral we have used the uniform boundedness of $(z_2)_t$, see the discussion before Step~1 of this proof. The uniform boundedness of $I_1$ follows analogously. Hence,  $\tilde H(y_t(t))$ is uniformly bounded in time. Furthermore, it can be seen that all the positive constants $C$ appearing in the above calculations depend continuously on the initial conditions. This concludes the proof.
\end{proof}

In order to extend this result to all classical solutions, we need the following density argument.

\begin{lem} \label{density}
For any $y \in D(\LL)$ there is a sequence $\{y_{n}\}_{n \in \N}$ in $D(\LL^2)$ such that $\lim_{n \to \infty}{y_n} =y$ and $\lim_{n \to \infty}{\LL y_n} = \LL y$.
\end{lem}

\begin{proof}
Let an arbitrary $y \in D(\LL)$ be fixed. Notice that it suffices to show that there exists a sequence $\{y_n\}_{n \in \N}$ with $y_n = [u_n \, v_n \, z_{1n} \, z_{2n} \, \xi_n \psi_n]^{\top}$ in $D(\LL^2)$ such that $\lim_{n \rightarrow \infty}{y_n} =y$ in the space $H^4(0,L) \times H^2(0,L) \times \R^{n_1} \times \R^{n_2} \times \R \times \R$.
The set $D(\LL^2)=\{y\in D(\LL):\LL y\in D(\LL)\}$ is equivalent to
\begin{align}
v &\in \tilde H_0^4(0,L),\label{night4}\\
u\in \tilde H_0^6(0,L) \,\wedge \, u^{\rom{4}}(0)&= u^{\rom{5}}(0)=0,\label{night5}\\
\xi&=Jv'(L),\label{night1}\\
\psi &=Mv(L),\label{night2}\\
\frac{\Lambda J}{\rho}u^{\rom{5}}(L)&= \Lambda u^{\prime\prime}(L) + \big[c_1(z_1) + d_1\Big(\frac{\xi}{J}\Big)+ k_1(u'(L))\big],\label{night3}\\
\frac{\Lambda M}{\rho}u^{\rom{4}}(L)&=-\Lambda u^{\prime\prime\prime}(L) + \big[c_2(z_2) 
+ d_2\Big(\frac{\psi}{M}\Big) + k_2(u(L))\big].\label{night6}
\end{align}
Since $\tilde{C}^{\infty}_0(0, L):=\{f \in C^{\infty}[0,L] \colon f^{(k)}(0) = 0, \forall k \in \N_0 \}$ is dense in $\tilde{H}^2_0(0,L)$ (see Theorem~3.17 in \cite{adams}), there exists a sequence $\{v_n\}_{n \in \mathbb{N}} \subset \tilde{C}^{\infty}_0(0, L)$ such that $\lim_{n \to\infty}{v_n} = v$ in $H^2(0,L)$. Also, $v_n$ satisfies \eqref{night4}, for all $n \in \N$. Defining $\xi_n := J v_n'(L)$ and $\psi_n := M v_n(L)$ ensures that $y_n$ satisfies \eqref{night1} and \eqref{night2}. Moreover, the Sobolev embedding $H^2(0,L) \inj C^1[0,L]$ implies that $\lim_{n \to\infty}{\xi_n} = \xi$ and $\lim_{n \to\infty}{\psi_n} = \psi$ as well. Next, let $z_{1n} := z_1$ and $z_{2n} := z_2$ for all $n \in \N$.

Finally, the sequence $\{ u_n \}_{n \in \mathbb{N}} \subset C^{\infty}[0,L]$ will be constructed such that $u_n$ satisfies \eqref{night5}, \eqref{night3}, and \eqref{night6} for all $n \in \mathbb{N}$, and $\lim_{n \to \infty}{u_n} = u$ in $H^4(0,L)$. To this end we introduce an auxiliary sequence of polynomial functions
\begin{align*}
h_n(x) &:= h_{2,n} x^2 + h_{3,n} x^3 + h_{6,n} x^6 + h_{7,n} x^7 + h_{8,n} x^8 \\
&\quad+ h_{9,n} x^9 + h_{10,n} x^{10} + h_{11,n} x^{11},
\end{align*}
for all $n \in \N$, where $h_{2,n}, \dots, h_{11,n} \in \R$ are to be determined. It immediately follows that
\begin{equation} \label{bc_0}
h_n(0) = h_n'(0) = h_n^{\rom{4}}(0) = h_n^{\rom{5}}(0) = 0.
\end{equation}  Let $h_{2,n} = \frac{u^{\prime\prime}(0)}{2}$ and $h_{3,n} = \frac{u^{\prime\prime\prime}(0)}{6}$, which is equivalent to
\begin{equation} \label{bc_00}
h_n^{\prime\prime}(0) = u^{\prime\prime}(0),\,\, h_n^{\prime\prime\prime}(0) = u^{\prime\prime\prime}(0).
\end{equation}
Further conditions are imposed on $h_n$:
\[\qquad h_n^{(k)}(L) = u^{(k)}(L), \qquad k \in \{0,1,2,3\}.\]
This can equivalently be written in terms of coefficients\footnote{The coefficient $k^{\underline{l}}$ (the Pochhammer symbol, see \cite{Knuth}) for $k,l \in \mathbb{N}$, $l \le k$ is defined by $k^{\underline{l}} := k\cdot (k-1) \cdots (k-l+1)$.}:
\begin{subequations} \label{cond1}
\begin{align}
r_1&=h_{6,n} + h_{7,n} L + h_{8,n} L^2 + h_{9,n} L^3
+ h_{10,n} L^{4} + h_{11,n} L^{5}, \\
r_2&= 6 h_{6,n} + 7 h_{7,n} L + 8 h_{8,n} L^2 + 9 h_{9,n} L^3 + 10 h_{10,n} L^{4} + 11 h_{11,n} L^{5},\\
r_3&= 6^{\underline{2}} h_{6,n} + 7^{\underline{2}} h_{7,n} L + 8^{\underline{2}} h_{8,n} L^2+ 9^{\underline{2}} h_{9,n} L^3  + 10^{\underline{2}} h_{10,n} L^{4} + 11^{\underline{2}} h_{11,n} L^{5}\\
r_4&=6^{\underline{3}} h_{6,n} + 7^{\underline{3}} h_{7,n} L + 8^{\underline{3}} h_{8,n} L^2+ 9^{\underline{3}} h_{9,n} L^3+ 10^{\underline{3}} h_{10,n} L^{4} + 11^{\underline{3}} h_{11,n} L^{5},
\end{align}
\end{subequations}
with
\begin{align*}
r_1 &=  \frac{u(L)}{L^6} - \frac{u^{\prime\prime}(0)}{2 L^4} - \frac{u^{\prime\prime\prime}(0)}{6 L^3},\qquad
r_2  =  \frac{u'(L)}{L^5} - \frac{u^{\prime\prime}(0)}{L^4} - \frac{u^{\prime\prime\prime}(0)}{2 L^3},\\
r_3  &=  \frac{u^{\prime\prime}(L)}{L^4} - \frac{u^{\prime\prime}(0)}{L^4} - \frac{u^{\prime\prime\prime}(0)}{L^3},\qquad
r_4  =  \frac{u^{\prime\prime\prime}(L)}{L^3} - \frac{u^{\prime\prime\prime}(0)}{L^3}.
\end{align*}
We further require that $h_n$ satisfies:
\begin{align}
\frac{\Lambda M }{\rho} h_n^{\rom{4}}(L) &= -\Lambda u^{\prime\prime\prime}(L) + \Big[c_2(z_2) +  d_2\Big(\frac{ \psi_n}{M}\Big) + k_2(u(L))\Big]=:r_5,\label{bc_1}\\
\frac{ \Lambda J}{\rho} h_n^{\rom{5}}(L) &= \Lambda u^{\prime\prime}(L) +\Big[c_1(z_1) + d_1\Big(\frac{\xi_n}{J}\Big)+ k_1(u'(L))\Big]:=r_6, \label{bc_11}
\end{align}
where \eqref{bc_1} and \eqref{bc_11} are equivalent to:
\begin{subequations} \label{cond2}
\begin{align}
 r_5\frac\rho{\Lambda M L^2}&=6^{\underline{4}} h_{6,n} + 7^{\underline{4}} h_{7,n} L + 8^{\underline{4}} h_{8,n} L^2 + 9^{\underline{4}} h_{9,n} L^3 + 10^{\underline{4}} h_{10,n} L^{4} + 11^{\underline{4}} h_{11,n} L^{5} ,\\
 r_6\frac\rho{\Lambda JL}&=6^{\underline{5}} h_{6,n} + 7^{\underline{5}} h_{7,n} L + 8^{\underline{5}} h_{8,n} L^2+ 9^{\underline{5}} h_{9,n} L^3 + 10^{\underline{5}} h_{10,n} L^{4} + 11^{\underline{2}} h_{11,n} L^{5} .
\end{align}
\end{subequations}
Such $h_n$ exists and is unique, due to the fact that linear system \eqref{cond1} and \eqref{cond2} has strictly positive determinant.
Consequently, \eqref{bc_0}, \eqref{bc_00}, and \eqref{cond1} imply that $u - h_n \in H^4_0(0,L)$, for all $n \in \N$. Since $C^{\infty}_0(0,L)$ is dense in $H^4_0(0,L)$, there exists a sequence $\{\tilde{u}_n\}_{n \in \mathbb{N}} \subset C^{\infty}_0(0,L)$ such that $\|\tilde{u}_n - (u - h_n)\|_{H^4} < \frac{1}{n}$, $\forall n \in \N$.  Now defining $u_n := \tilde{u}_n + h_n$, gives $\lim_{n \to \infty}{u_n} = u$ in $H^4(0,L)$. Obviously $u_n$ satisfies \eqref{night5} for all $n \in \mathbb{N}$. Also, due to \eqref{bc_1} and \eqref{bc_11}, $u_n$ satisfies \eqref{night3} and \eqref{night6}, as well. Hence, the statement follows.
\end{proof}

\begin{trm} \label{precompactness}
For all $y_0 \in D({\LL})$ the trajectory $\gamma(y_0)$ is precompact in $\HH$.
\end{trm}
\begin{proof}
Let $y_0 \in D(\LL)$ be chosen arbitrarily, and let $\{y_{n0}\}_{n \in \N} \subset D(\LL^2)$ be an approximating sequence as in Lemma~\ref{density}. Then there holds:
\begin{equation} \label{stronger_conv}
\lim_{n \to \infty}{\LL y_{n0}}=\LL y_0.
\end{equation}
For an arbitrary $T>0$, and by applying Proposition~\ref{prp_2} it follows that the approximating solutions $y_n(t)$ converge to $y(t)$ in $C([0,T]; \HH)$.  Since $y_n(t) \in C^1([0, \infty);\HH)$ and solves \eqref{ivp} for all $n \in \N$, \eqref{stronger_conv} yields
\begin{equation} \label{derivative_conv}
\lim_{n \to\infty}{(y_{n})_t(0)} = \LL y_0 \,\, \text{ in } \,\, \HH.
\end{equation}
Hence, \eqref{unif_bound} and \eqref{derivative_conv} imply that there exists a constant $C>0$ such that for all $n \in \N$:
\[
\sup_{t\ge0}{\|(y_n)_t(t)\|_{\HH}} \le C(\|y_0\|_{\HH}, \|\LL y_0\|_{\HH}),
\]
where the constant $C$ does not depend on $n$.  From here it follows that $(y_n)_t$ is bounded in $L^{\infty}(\R^+ ; \HH)$. Hence, the Banach-Alaoglu Theorem (see Theorem~I.3.15 in \cite{Rudin}) implies that there exists $w \in L^{\infty}(\R^+ ; \HH)$ and a subsequence $\{y_{n_k}\}_{k \in \N}$ such that
\[ (y_{n_k})_t \stackrel{\ast}{\rightharpoonup} w \,\, \text{ in } \,\,L^{\infty}(\R^+ ; \HH).\] For arbitrary $z \in \HH$ and $t \ge 0$ there holds
\[\lim_{k \to \infty}{\int_0^t \la (y_{n_k})_t(\tau) , z \ra_{\HH} \d{}\tau} = \int_0^t \la w(\tau), z \ra_{\HH} \d{}\tau,\] which is equivalent to
\[\lim_{k \to \infty}  \la y_{n_k}(t) - y_{n_k}(0) , z \ra_{\HH}  =  \Big\la \int_0^t w(\tau) \d{}\tau, z\Big \ra_{\HH}. \] Since $\lim_{n \to \infty}y_n(\tau) = y(\tau)$ (in $\HH$) for all $\tau \in [0, \infty)$, it follows that
\[\la y(t) - y(0) , z \ra_{\HH}  = \Big\la \int_0^t w(\tau) \d{} \tau, z \Big\ra_{\HH}.\]
Since $z\in  \HH$ was arbitrary, we obtain
\begin{equation} \label{der}
y(t) - y(0) = \int_0^t w(\tau) \d{} \tau, \quad\forall t\ge 0.
\end{equation}
Due to continuous differentiability of $y$, the time derivative of \eqref{der} can be taken, which yields $y_t \equiv w$. This implies $y_t \in L^{\infty}(\R^+ ; \HH)$, i.e.~$\|y_t(\cdot)\|_{\HH}$ is uniformly bounded, which proves the theorem.
\end{proof}

\begin{cor}
For any $y_0\in D(\LL)$ there holds $\lim_{t\to\infty}=S(t)y_0=0$.
\end{cor}

\section{Conclusions}

In this paper, we provide a rigorous stability proof of a lossless Euler-Bernoulli beam with tip mass which is feedback interconnected with a nonlinear spring-damper system and a strictly passive nonlinear dynamical system.  Such a configuration comes into play if the tip payload is interacting with a nonlinear passive environment, if the (nonlinear) dynamics of the torque and force actuators are also taken into account, or for a combination of these cases. It is well known that the feedback interconnection of passive systems is passive with a storage function that is the sum of the storage functions of all subsystems. In the finite-dimensional case, this property is advantageously utilized for the controller design where the storage function usually qualifies as an appropriate Lyapunov function candidate. For the infinite-dimensional system under consideration, the passivity property still ensures that the storage functional is non-increasing along classical solutions, however, it is well known that this does not directly entail asymptotic stability. In fact, a crucial step in the stability analysis is to prove the precompactness of the trajectories. For linear evolution problems this has been reported in many contributions in the literature, but when considering nonlinearities this is much more involved. Under rather mild conditions on the parameters and functions appearing in the resulting PDE--ODE model representing the overall closed-loop system, global-in-time wellposedness is proven by means of semigroup theory and the precompactness of the trajectories is shown by deriving uniform-in-time bounds on the solution and its time derivatives. With this, asymptotic stability of classical solutions can be guaranteed.


\appendix
\section{The Operator ${\A}_p$}

The system \eqref{int_verss} is the  mild formulation of the evolution problem $(y_p)_t={\A}_py_p$ with $y_p=[u,v]^\top\in \HH_p$. Thereby $\HH_p:=\tilde H_0^2(0,L)\times L^2(0,L)$, and \[{\A}_p:\begin{bmatrix}
                       u\\ v
                  \end{bmatrix}\mapsto\begin{bmatrix} v\\-\frac\Lambda\rho u^\rom{4}\end{bmatrix},\]
with the domain
\begin{align*}
D({\A}_p)=\big\{[u,v]^\top\!\!\in\HH_p&\colon u\in \tilde H_0^4(0,L), v\in \tilde H_0^2(0,L),  \\
&\phantom{\colon}\Lambda u^{\prime\prime}(L)+\tilde K_1 u'(L)= 0, \Lambda u^{\prime\prime\prime}(L)-\tilde K_2 u(L)=0\big\}.
\end{align*}

The space $\HH_p$ is equipped with the following inner product:
\begin{equation}\label{App1}
\begin{split}
\la y_p,\tilde y_p\ra_p &:= {\Lambda} \int_{0}^{L}{u^{\prime\prime} \tilde u^{\prime\prime} \d x} + {\rho} \int_{0}^{L}{v\tilde v \d x}\\
&\quad\quad+ \tilde K_1 u'(L) \tilde u'(L) + \tilde K_2 u(L) \tilde u(L).
\end{split}
\end{equation}

The constants $\tilde K_1,\tilde K_2$ are defined in \eqref{5_star} and depend, at first glance, on the fixed $y_0\in \Omega$ in the proof of Theorem~\ref{om_lise_0}. Hence, $D(\LL_p)$ and the above inner product also depend on $y_0$. But this does not cause any problems. Anyhow, Step 2 in the proof of Theorem~\ref{om_lise_0} shows that $u_0(L)=u_0'(L)=0$. Hence, $\tilde K_j=K_j$.

We have the following results:

\begin{lem}\label{invertdf}
The {operator ${\A}_p^{-1}:\HH_p\to D({\A}_p)$} exists and is a bijection. Furthermore, ${\A}^{-1}_p$ is compact in $\HH_p$.
\end{lem}

\begin{proof}
The proof is analogous to the proof of Theorem~\ref{teo73}, see also Section 4.2 in \cite{KT05}.
\end{proof}

\begin{lem}\label{skew}
 The operator ${\A}_p$ is skew-adjoint.
\end{lem}

\begin{proof}
First we show that ${\A}_p$ is skew-symmetric, i.e.~for all $y,\tilde y\in D({\A}_p)$ there holds $\la {\A}_p y, \tilde y\ra_p=-\la y,{\A}_p\tilde y\ra_p$:
\begin{align*}
 \la {\A}_p y, \tilde y\ra_p &= \Lambda\int_0^L v^{\prime\prime}\tilde u^{\prime\prime}\d x-\Lambda \int_0^L u^\rom{4}\tilde v\d x +\tilde K_1v'(L)\tilde u'(L)+\tilde K_2v(L)\tilde u(L)\\
&=\Lambda\Big(\int_0^L v\tilde u^\rom{4}\d x+ v'(L)\tilde u^{\prime\prime}(L)-v(L)\tilde u^{\prime\prime\prime}(L)\\
&\quad- \int_0^L u^{\prime\prime}\tilde v^{\prime\prime}\d x-u^{\prime\prime\prime}(L)\tilde v(L)+u^{\prime\prime}(L)\tilde v'(L)\Big)\\
&\phantom{=\,}+\tilde K_1v'(L)\tilde u'(L)+\tilde K_2v(L)\tilde u(L).
\end{align*}
Using the boundary conditions $ \Lambda u^{\prime\prime}(L)+\tilde K_1 u'(L)=0$ and $\Lambda u^{\prime\prime\prime}(L)-\tilde K_2 u(L)=0$ from $D({\A}_p)$ we obtain:
\begin{align*}
  \la {\A}_p y, \tilde y\ra_p &=\Lambda\int_0^L v\tilde u^\rom{4}\d x- \tilde K_1v'(L)\tilde u'(L)-\tilde K_2v(L)\tilde u(L)\\
  &\quad- \Lambda\!\int_0^L\!\!\! u^{\prime\prime}\tilde v^{\prime\prime}\d x-\tilde K_2u(L)\tilde v(L)-\tilde K_1u'(L)\tilde v'(L)\\
  &\quad+\tilde K_1v'(L)\tilde u'(L)+\tilde K_2v(L)\tilde u(L)\\
&=-\la y,{\A}_p\tilde y\ra_p.
\end{align*}
So ${\A}_p$ is skew-symmetric. Furthermore, due to Lemma~\ref{invertdf} we know that $\ran {\A}_p=\HH_p$. So we can apply the Corollary of Theorem~VII.3.1 in \cite{yosida}, which proves the skew-adjointness of ${\A}_p$.
\end{proof}

\begin{lem}\label{ston}
 ${\A}_p$ generates a $C_0$-semigroup of unitary operators in $\HH_p$.
\end{lem}

\begin{proof}
 Since ${\A}_p$ is skew-adjoint, this follows from Stone's theorem \cite[Theorem~II.3.24]{engel}.
\end{proof}

\section*{Acknowledgment}

This research was supported by the FWF-doctoral school ``Dissipation
and dispersion in nonlinear partial differential equations'' and the
FWF-project I395-N16.  Two authors (AA, MM) acknowledge a sponsorship by \emph{Clear Sky Ventures}.



\begin{thebibliography}{10}
\providecommand{\url}[1]{#1}
\csname url@samestyle\endcsname
\providecommand{\newblock}{\relax}
\providecommand{\bibinfo}[2]{#2}
\providecommand{\BIBentrySTDinterwordspacing}{\spaceskip=0pt\relax}
\providecommand{\BIBentryALTinterwordstretchfactor}{4}
\providecommand{\BIBentryALTinterwordspacing}{\spaceskip=\fontdimen2\font plus
\BIBentryALTinterwordstretchfactor\fontdimen3\font minus
  \fontdimen4\font\relax}
\providecommand{\BIBforeignlanguage}[2]{{%
\expandafter\ifx\csname l@#1\endcsname\relax
\typeout{** WARNING: IEEEtran.bst: No hyphenation pattern has been}%
\typeout{** loaded for the language `#1'. Using the pattern for}%
\typeout{** the default language instead.}%
\else
\language=\csname l@#1\endcsname
\fi
#2}}
\providecommand{\BIBdecl}{\relax}
\BIBdecl

\bibitem{littman}
W.~Littman and L.~Markus, ``Stabilization of a hybrid system of elasticity by
  feedback boundary damping,'' \emph{Ann. Mat. Pura Appl. (4)}, vol. 152, pp.
  281--330, 1988.

\bibitem{Morgul2001}
{\"O}.~Morg{\"u}l, ``{Stabilization and Disturbance Rejection for the Beam
  Equation},'' \emph{IEEE Transactions on Automatic Control}, vol.~46, no.~12,
  pp. 1913--1918, 2001.

\bibitem{guo2002riesz}
B.-Z. Guo, ``Riesz basis property and exponential stability of controlled
  euler--bernoulli beam equations with variable coefficients,'' \emph{SIAM
  Journal on Control and Optimization}, vol.~40, no.~6, pp. 1905--1923, 2002.

\bibitem{guo2006riesz}
B.-Z. Guo and J.-M. Wang, ``Riesz basis generation of abstract second-order
  partial differential equation systems with general non-separated boundary
  conditions,'' \emph{Numerical functional analysis and optimization}, vol.~27,
  no. 3-4, pp. 291--328, 2006.

\bibitem{rao1995uniform}
B.~Rao, ``Uniform stabilization of a hybrid system of elasticity,'' \emph{SIAM
  Journal on Control and Optimization}, vol.~33, no.~2, pp. 440--454, 1995.

\bibitem{conrad1998stabilization}
F.~Conrad and {\"O}.~Morg{\"u}l, ``On the stabilization of a flexible beam with
  a tip mass,'' \emph{SIAM Journal on Control and Optimization}, vol.~36,
  no.~6, pp. 1962--1986, 1998.

\bibitem{lgm}
Z.-H. Luo, B.-Z. Guo, and {\"O}.~Morg{\"u}l, \emph{Stability and stabilization
  of infinite dimensional systems with applications}, ser. Communications and
  Control Engineering Series.\hskip 1em plus 0.5em minus 0.4em\relax London:
  Springer, 1999.

\bibitem{JakobZwart2012}
B.~Jakob and H.~Zwart, \emph{Linear Port-Hamiltonian Systems on
  Infinite-dimensional Spaces}, ser. Operator Theory: Advances and
  Applications.\hskip 1em plus 0.5em minus 0.4em\relax Basel: Birkhauser, 2012,
  vol. 223.

\bibitem{schaft2000}
A.~van~der Schaft, \emph{L2-gain and passivity techniques in nonlinear
  control}, 2nd~ed., ser. Communications and Control Engineering.\hskip 1em
  plus 0.5em minus 0.4em\relax London: Springer, 2000.

\bibitem{ortega2002interconnection}
R.~Ortega, A.~van~der Schaft, B.~Maschke, and G.~Escobar, ``Interconnection and
  damping assignment passivity-based control of port-controlled hamiltonian
  systems,'' \emph{Automatica}, vol.~38, no.~4, pp. 585--596, 2002.

\bibitem{ott2008passivity}
C.~Ott, A.~Albu-Schaffer, A.~Kugi, and G.~Hirzinger, ``On the passivity-based
  impedance control of flexible joint robots,'' \emph{IEEE Transactions on
  Robotics}, vol.~24, no.~2, pp. 416--429, 2008.

\bibitem{KT05}
A.~Kugi and D.~Thull, ``{Infinite-Dimensional Decoupling Control of the Tip
  Position and the Tip Angle of a Composite Piezoelectric Beam with Tip
  Mass},'' in \emph{Control and Observer Design for Nonlinear Finite and
  Infinite Dimensional Systems}, T.~Meuer, K.~Graichen, and E.~D. Gilles,
  Eds.\hskip 1em plus 0.5em minus 0.4em\relax Berlin Heidelberg: Springer,
  2005, pp. 351--368.

\bibitem{Villegas2009}
J.~A. Villegas, H.~Zwart, Y.~Le~Gorrec, and B.~Maschke, ``Exponential stability
  of a class of boundary control systems,'' \emph{IEEE Transactions on
  Automatic Control}, vol.~54, no.~1, pp. 142--147, 2009.

\bibitem{Ramirez2014}
H.~Ramirez, Y.~Le~Gorrec, A.~Macchelli, and H.~Zwart, ``Exponential
  stabilization of boundary controlled port-hamiltonian systems with dynamic
  feedback,'' \emph{IEEE Transactions on Automatic Control}, vol.~59, no.~10,
  pp. 2849--2855, 2014.

\bibitem{MA}
M.~Mileti\'c and A.~Arnold, ``A piezoelectric {E}uler-{B}ernoulli beam with
  dynamic boundary control: Stability and dissipative {FEM},'' \emph{Acta
  Applicandae Mathematicae}, pp. 1--37, 2014.

\bibitem{cra_pa}
M.~G. Crandall and A.~Pazy, ``Semi-groups of nonlinear contractions and
  dissipative sets,'' \emph{J. Functional Analysis}, vol.~3, pp. 376--418,
  1969.

\bibitem{daf_sle}
C.~M. Dafermos and M.~Slemrod, ``Asymptotic behavior of nonlinear contraction
  semigroups,'' \emph{J. Functional Analysis}, vol.~13, pp. 97--106, 1973.

\bibitem{cor_nov}
J.-M. Coron and B.~d'Andrea Novel, ``Stabilization of a rotating body beam
  without damping,'' \emph{Automatic Control, IEEE Transactions on}, vol.~43,
  no.~5, pp. 608--618, 1998.

\bibitem{Pazy3}
A.~Pazy, ``A class of semi-linear equations of evolution,'' \emph{Israel J.
  Math.}, vol.~20, pp. 23--36, 1975.

\bibitem{Webb}
G.~F. Webb, ``Compactness of bounded trajectories of dynamical systems in
  infinite-dimensional spaces,'' \emph{Proc. Roy. Soc. Edinburgh Sect. A},
  vol.~84, no. 1-2, pp. 19--33, 1979.

\bibitem{zuyev}
A.~L. Zuyev, \emph{Partial stabilization and control of distributed parameter
  systems with elastic elements}, ser. Lecture Notes in Control and Information
  Sciences.\hskip 1em plus 0.5em minus 0.4em\relax Springer, Cham, 2015, vol.
  458.

\bibitem{MSA}
M.~Mileti\'c, D.~St{\"u}rzer, and A.~Arnold, ``{An {E}uler-{B}ernoulli beam
  with nonlinear damping and a nonlinear spring at the tip},'' \emph{arXiv
  preprint: 1411.7946}, 2014.

\bibitem{adams}
R.~A. Adams and J.~J.~F. Fournier, \emph{Sobolev spaces}, 2nd~ed., ser. Pure
  and Applied Mathematics.\hskip 1em plus 0.5em minus 0.4em\relax Amsterdam:
  Elsevier/Academic Press, 2003, vol. 140.

\bibitem{lozano2000}
R.~Lozano, B.~Brogliato, O.~Egeland, and B.~Maschke, \emph{Dissipative Systems
  Analysis and Control}.\hskip 1em plus 0.5em minus 0.4em\relax London:
  Springer, 2000.

\bibitem{pazy}
A.~Pazy, \emph{Semigroups of linear operators and applications to partial
  differential equations}, ser. Applied Mathematical Sciences.\hskip 1em plus
  0.5em minus 0.4em\relax New York: Springer, 1983, vol.~44.

\bibitem{Cazenave:Haraux}
T.~Cazenave and A.~Haraux, \emph{An introduction to semilinear evolution
  equations}, ser. Oxford Lecture Series in Mathematics and its
  Applications.\hskip 1em plus 0.5em minus 0.4em\relax New York: The Clarendon
  Press Oxford University Press, 1998, vol.~13.

\bibitem{thesis_DS}
D.~St{\"u}rzer, ``Stability of an {E}uler-{B}ernoulli beam coupled to nonlinear
  control systems,'' Ph.D. dissertation, Vienna University of Technology, 2015.

\bibitem{nirenberg}
L.~Nirenberg, ``On elliptic partial differential equations,'' \emph{Ann. Scuola
  Norm. Sup. Pisa (3)}, vol.~13, pp. 115--162, 1959.

\bibitem{john}
F.~John, \emph{Partial differential equations}, 4th~ed., ser. Applied
  Mathematical Sciences.\hskip 1em plus 0.5em minus 0.4em\relax New York:
  Springer, 1982, vol.~1.

\bibitem{Knuth}
R.~L. Graham, D.~E. Knuth, and O.~Patashnik, \emph{Concrete Mathematics}.\hskip
  1em plus 0.5em minus 0.4em\relax Massachusetts: Addison-Wesley, 1989.

\bibitem{Rudin}
W.~Rudin, \emph{Functional analysis}, 2nd~ed., ser. International series in
  pure and applied mathematics.\hskip 1em plus 0.5em minus 0.4em\relax New
  York: McGraw-Hill, 1991.

\bibitem{yosida}
K.~Yosida, \emph{Functional analysis}, 6th~ed., ser. Grundlehren der
  Mathematischen Wissenschaften.\hskip 1em plus 0.5em minus 0.4em\relax Berlin:
  Springer, 1980, vol. 123.

\bibitem{engel}
K.-J. Engel and R.~Nagel, \emph{One-parameter semigroups for linear evolution
  equations}, ser. Graduate Texts in Mathematics.\hskip 1em plus 0.5em minus
  0.4em\relax New York: Springer, 2000, vol. 194.

\end{thebibliography}
\end{document}